\documentclass[11pt,reqno]{amsart}
\usepackage{amssymb,amsthm}
\usepackage[nosumlimits]{mathtools}
\usepackage{enumitem}
\usepackage[margin=1in]{geometry}
\usepackage{eucal}
\usepackage{xfrac,multirow}
\usepackage{array}

\theoremstyle{plain}
\newtheorem{theorem}{Theorem}[section]
\newtheorem{proposition}[theorem]{Proposition}
\newtheorem{lemma}[theorem]{Lemma}
\newtheorem{corollary}[theorem]{Corollary}

\theoremstyle{definition}
\newtheorem{definition}[theorem]{Definition}

\usepackage{hyperref}
\usepackage{xcolor}
\hypersetup{
    colorlinks,
    linkcolor={red!50!black},
    citecolor={blue!50!black},
    urlcolor={blue!80!black}
}

\newcommand{\verteq}{\rotatebox{90}{$\,=$}}
\newcommand{\equalto}[2]{\underset{\displaystyle\overset{\mkern4mu\verteq}{#2}}{#1}}

\usepackage{stmaryrd}
\newcommand{\lb}{\llbracket}
\newcommand{\rb}{\rrbracket}

\newcommand{\tp}{{\scriptscriptstyle\mathsf{T}}}
\newcommand{\h}{{\scriptscriptstyle\mathsf{H}}}
\newcommand{\G}{{\scriptscriptstyle\mathsf{G}}}
\newcommand{\La}{{\scriptscriptstyle\mathsf{L}}}
\newcommand{\J}{{\scriptscriptstyle\mathsf{J}}}
\newcommand{\C}{{\scriptscriptstyle\mathsf{C}}}
\newcommand{\p}{{\scriptscriptstyle+}}
\newcommand{\x}{{\scriptscriptstyle\times}}
\newcommand{\s}{{\scriptscriptstyle\mathsf{S}}}
\newcommand{\la}{{\scriptscriptstyle\mathsf{L}}}
\let\O\undefined
\DeclareMathOperator{\O}{O}
\DeclareMathOperator{\U}{U}
\DeclareMathOperator{\V}{V}
\DeclareMathOperator{\Sp}{Sp}
\DeclareMathOperator{\SO}{SO}
\DeclareMathOperator{\tr}{tr}
\DeclareMathOperator{\Gr}{Gr}
\DeclareMathOperator{\LGr}{LGr}
\DeclareMathOperator{\diag}{diag}

\DeclareMathOperator{\spn}{span}

\DeclareMathOperator{\Flag}{Flag}
\DeclareMathOperator{\goe}{GOE}
\DeclareMathOperator{\gue}{GUE}
\DeclareMathOperator{\gse}{GSE}
\DeclareMathOperator{\unif}{\textsc{Unif}}

\DeclareMathOperator{\cre}{CRE}
\DeclareMathOperator{\cue}{CUE}
\DeclareMathOperator{\cqe}{CQE}
\DeclareMathOperator{\coe}{COE}
\DeclareMathOperator{\cle}{CLE}
\DeclareMathOperator{\cse}{CSE}
\DeclareMathOperator{\gio}{\textsc{GinOE}}
\DeclareMathOperator{\giu}{\textsc{GinUE}}
\DeclareMathOperator{\gis}{\textsc{GinSE}}
\DeclareMathOperator{\joe}{JOE}
\DeclareMathOperator{\loe}{LOE}
\DeclareMathOperator{\jue}{JUE}
\DeclareMathOperator{\lue}{LUE}
\DeclareMathOperator{\jse}{JSE}
\DeclareMathOperator{\lse}{LSE}
\DeclareMathOperator{\nor}{N}

\let\latexcirc=\circ
\newcommand{\ccirc}{\mathbin{\mathchoice
  {\xcirc\scriptstyle}
  {\xcirc\scriptstyle}
  {\xcirc\scriptscriptstyle}
  {\xcirc\scriptscriptstyle}
}}
\newcommand{\xcirc}[1]{\vcenter{\hbox{$#1\latexcirc$}}}
\let\circ\ccirc

\begin{document}
\title[Random Eigen, singular, cosine-sine, and Autonne--Takagi vectors]{Eigen, singular, cosine-sine, and Autonne--Takagi vectors distributions of random matrix ensembles}
\author[Y.~Guo]{Yihan~Guo}
\address{Computational and Applied Mathematics Initiative, Department of Statistics,
University of Chicago, Chicago, IL 60637-1514.}
\email{yihanguo@uchicago.edu}
\author[L.-H.~Lim]{Lek-Heng~Lim}
\address{Computational and Applied Mathematics Initiative, Department of Statistics,
University of Chicago, Chicago, IL 60637-1514.}
\email{lekheng@uchicago.edu}

\begin{abstract}
We show that some of the best-known matrix decompositions of some of the best-known random matrix ensembles give us the unique $G$-invariant uniform distributions on some of the best-known manifolds. The eigenvectors distributions of the Gaussian, Laguerre, and Jacobi ensembles are all given by the uniform distribution on the complete flag manifold. The singular vectors distributions of Ginibre ensembles are given by the uniform distribution on a product of the complete flag manifold with a Stiefel manifold. Circular ensembles split into two types: The cosine-sine vectors distributions of circular real, unitary, and quaternionic ensembles are given by the uniform distributions on products of a (partial) flag manifold with copies of the orthogonal, unitary, or compact symplectic groups. The Autonne--Takagi vectors distributions of circular orthogonal, Lagrangian, and symplectic ensembles are given by the uniform distributions on Lagrangian Grassmannians.
\end{abstract}

\maketitle

\section{Introduction}\label{sec:intro}
A flag manifold $\Flag(k_1,\dots,k_p, \mathbb{R}^n)$ is the set of all $(k_1,\dots,k_p)$-flags, i.e., nested sequences of subspaces of dimensions $k_1 < \dots < k_p$, in $\mathbb{R}^n$. It has two noteworthy boundary cases: The $p = 1$ case is the Grassmannian or Grassmann manifold, the set of all $k$-planes, i.e., $k$-dimensional subspaces, in $\mathbb{R}^n$. The $p = n-1$ case is the \emph{complete flag manifold}  where $k_1 = 1$, $k_2 = 2, \dots, k_{n-1} =  n-1$. These are given special notations:
\[
\Gr(k, \mathbb{R}^n) \coloneqq \Flag(k, \mathbb{R}^n), \qquad \Flag(\mathbb{R}^n) \coloneqq \Flag(1,2,\dots,n-1,\mathbb{R}^n).
\]
Closely related is the Stiefel manifold $\V(k, \mathbb{R}^n)$ of $k$-frames, i.e., ordered orthonormal basis of $k$-planes, in $\mathbb{R}^n$, which will also play a role in this article.

The set-theoretic description above defines the flag manifold as an abstract manifold. However, one may also characterize it as a homogeneous space:
\[
\Flag(k_1,\dots,k_p,\mathbb{R}^n) \cong \O(n)/\bigl(\O(k_1)\times\O(k_2-k_1)\times\dots\times\O(k_p-k_{p-1})\times\O(n-k_p)\bigr),
\]
or as an embedded submanifold:
\begin{equation}\label{eq:flag}
\Flag(k_1,\dots,k_p,\mathbb{R}^n) \cong \{Q\Delta_aQ^\tp \in \mathsf{S}^2(\mathbb{R}^n) :Q\in\O(n)\}.
\end{equation}
Here $\Delta_a  \coloneqq \diag(a_1I_{k_1-k_0},a_2I_{k_2-k_1},\dots,a_{p+1}I_{k_{p+1}-k_p})$ is a diagonal matrix with arbitrary, fixed, and distinct $a_1, a_2,\dots,a_{p+1} \in \mathbb{R}$ appearing with multiplicities $k_1-k_0,k_2-k_1,\dots,k_{p+1} - k_p$ where we set $k_0 \coloneqq 0$ and $k_{p+1} \coloneqq n$. We denote the space of $n \times n$ real symmetric matrices by $\mathsf{S}^2(\mathbb{R}^n)$.
The homogeneous space structure endows $\Flag(k_1,\dots,k_p,\mathbb{R}^n)$ with a unique $\O(n)$-invariant probability measure (see Section~\ref{sec:back}). The submanifold structure allows points on $\Flag(k_1,\dots,k_p,\mathbb{R}^n)$ to be uniquely represented as actual matrices (as opposed to equivalence classes of matrices).

A world apart from the flag manifold is the Gaussian Orthogonal Ensemble (GOE), a well-known probability distribution on $\mathsf{S}^2(\mathbb{R}^n)$ with density
\[
f_{\goe(n)}(X)=2^{-\frac{n}{2}}\pi^{-\frac{n(n+1)}{4}}e^{-\frac{1}{2}\tr(X^2)}.
\]
Write $X = Q \Lambda Q^\tp$ with $Q \in \O(n)$ and $\Lambda \coloneqq \diag(\lambda_1, \dots,\lambda_n)$ for an eigenvalue decomposition of $X \in  \mathsf{S}^2(\mathbb{R}^n)$ with eigenvalues ordered $\lambda_1 > \dots > \lambda_n$ (or any other fixed prescribed order). For $X \sim \goe(n)$, both $\Lambda$ and $Q$ are random matrices. The distribution of $\Lambda$ has been thoroughly studied. The distribution of $Q$ has however received much less attention. A common belief is that there is not much to study --- the latter is simply the uniform distribution on $\O(n)$, its normalized Haar measure, see \cite[p.~1623]{inaccurate2} and \cite[p.~289]{inaccurate1} for example. This is not quite right. An indication that it is inaccurate is that $Q$ is evidently not unique, even when the eigenvalues are distinct and ordered, and all eigenvectors normalized to have unit norm. One may enforce uniqueness on the matrix of eigenvectors by further requiring that the first row of $Q$  have all entries positive, as was done in \cite[p.~367]{accurate1}; but this is a coordinate-dependent remedy, and describes the GOE eigenvectors distribution on a strictly smaller subset $\{ Q \in \O(n) : q_{1j} > 0, \; j =1,\dots,n\}$ that is not even a subgroup.\footnote{This issue was also raised in \url{https://mathoverflow.net/a/413187} but with no resolution. Our work here provides an answer: The eigenvectors should be regarded as points on a flag manifold, not the unitary group.}

We will show in this article that the space underlying this distribution is in fact the complete flag manifold $\Flag(\mathbb{R}^n)$, and the GOE eigenvectors distribution its unique $\O(n)$-invariant uniform probability distribution. The treatment in \cite[p.~4]{univpropvec} gets close to this picture but stops short of mentioning the complete flag manifold. 
We highlight two departures from existing works:
\begin{enumerate}[{label=\upshape(\roman*)}]
\item The issue above does not happen to the widely known QR decomposition of a random matrix from the Ginibre Ensemble \cite{GinibreQR}, i.e., an $n \times n$ matrix of i.i.d.\ standard normal variates. In this case the distribution of the $Q$-factor is indeed the uniform distribution on $\O(n)$. The reason is that the $Q$-factor in a QR decomposition $A = QR$ of a generic matrix $A$ is unique as long as we restrict $R$ to have all diagonal entries positive (or negative). On the other hand, the $Q$-factor in an eigenvalue decomposition $X = Q \Lambda Q^\tp$ of a generic symmetric matrix $X$ cannot be unique no matter what restriction we impose on $\Lambda$. The typical way to ensure uniqueness is to impose restrictions on $Q$ itself; but once we do this, as we saw, we will no longer obtain a distribution defined on all of $\O(n)$.

\item\label{it:two} Our plural nomenclature, i.e., \emph{eigenvectors} distribution, is deliberate. Existing studies about the \emph{eigenvector} distribution of GOE, or, more generally, of Wigner matrices, are limited to one eigenvector, i.e., a column of $Q$ \cite{univpropvec2,univpropvec}, or even one entry of $Q$ \cite{univpropvec}. Here we study the distribution of the eigenbasis, i.e., the complete set of eigenvectors assembled in a matrix $Q$. It turns out that it is easier to study the whole $Q$ than it is to study its column vectors or its entries in isolation. By using the characterization of flag manifold as a set of matrices in \eqref{eq:flag}, the matrix of eigenvectors gives a well-defined map $\mathsf{S}^2(\mathbb{R}^n)  \to \Flag(\mathbb{R}^n)$ that takes $\goe(n)$ on the former to the uniform distribution on the latter. Indeed, we will go further to show that one may obtain the unique $\O(n)$-invariant uniform probability distribution on \emph{any} flag manifold $\Flag(k_1,\dots,k_p, \mathbb{R}^n)$ in much the same way.
\end{enumerate}

Unsurprisingly, we will see that \ref{it:two} applies nearly verbatim when $\mathbb{R}$ is replaced by $\mathbb{C}$ or $\mathbb{H}$: The eigenvectors distribution of the Gaussian Unitary Ensemble (GUE) yields the $\U(n)$-invariant uniform probability distribution on any complex flag manifold $\Flag(k_1,\dots,k_p, \mathbb{C}^n)$. Likewise the eigenvectors distribution of the Gaussian Symplectic Ensemble (GSE)  yields the $\Sp(n)$-invariant uniform probability distribution on any quaternionic flag manifold $\Flag(k_1,\dots,k_p, \mathbb{H}^n)$. These will be discussed alongside \ref{it:two} in Section~\ref{sec:gauss} and then extended to the Laguerre and Jacobi ensembles.

In Section~\ref{sec:gini}, we will similarly study the singular vectors distributions of Ginibre ensembles over $\mathbb{R}$, $\mathbb{C}$, and $\mathbb{H}$, known in the literature by the names Ginibre Orthogonal Ensemble (\textsc{GinOE}), Ginibre Unitary Ensemble (\textsc{GinUE}), and Ginibre Symplectic Ensemble (\textsc{GinSE}) respectively. These would also yield the uniform distributions on the real, complex, and quaternionic Stiefel manifolds as a side product.

We conclude with two sections looking into two kinds of ensembles of quite distinct nature but both falling under the heading of ``circular ensembles.'' In Section~\ref{sec:circ}, we discuss the cosine-sine vectors distributions of the Circular Real Ensemble (CRE), Circular Unitary Ensemble (CUE), and Circular Quaternion Ensemble (CQE). In Section~\ref{sec:circ2}, we discuss the Autonne--Takagi vectors distributions of the Circular Orthogonal Ensemble (COE), Circular Lagrangian Ensemble (CLE),  and Circular Symplectic Ensemble (CSE). What we call ``Circular Lagrangian Ensemble'' does not currently have a name in the existing literature; we will see in Section~\ref{sec:circ2} why this is a fitting name.

We summarize our findings in Table~\ref{tab:sum}. For the uninitiated, all aforementioned terms will be properly defined in Section~\ref{sec:back}, complete with pointers to the relevant references.
\begin{table}[h]
\centering
\footnotesize
\caption{Eigen, singular, and cosine-sine vectors distributions of random matrix ensembles}
\label{tab:sum}
\renewcommand{\arraystretch}{1.3}
\begin{tabular}{l|l|l}
\textsc{ensemble} & \textsc{object} & \textsc{distribution} \\
\hline
$\goe(n)$/$\loe(m,n)$/$\joe(l,m,n)$ & eigenvectors & $\unif\bigl(\Flag(\mathbb{R}^n)\bigr)$ \\
$\gue(n)$/$\lue(m,n)$/$\jue(l,m,n)$ & eigenvectors & $\unif\bigl(\Flag(\mathbb{C}^n)\bigr)$ \\
$\gse(n)$/$\lse(m,n)$/$\jse(l,m,n)$ & eigenvectors & $\unif\bigl(\Flag(\mathbb{H}^n)\bigr)$\\
$\gio(m,n)$ & singular vectors & $\unif\bigl(\V(n,\mathbb{R}^m)\times\Flag(\mathbb{R}^n)\bigr)$ \\
$\giu(m,n)$ & singular vectors & $\unif\bigl(\V(n,\mathbb{C}^m)\times\Flag(\mathbb{C}^n)\bigr)$\\
$\gis(m,n)$ & singular vectors & $\unif\bigl(\V(n,\mathbb{H}^m)\times\Flag(\mathbb{H}^n)\bigr)$\\
$\cre(n)$ & cosine-sine vectors & $\unif\bigl(\O(k)^2\times\O(n-k)\times\Flag(1,\dots,k,\mathbb{R}^{n-k})\bigr)$\\
$\cue(n)$ & cosine-sine vectors& $\unif\bigl(\U(k)^2\times\U(n-k)\times\Flag(1,\dots,k,\mathbb{C}^{n-k})\bigr)$ \\
$\cqe(n)$ & cosine-sine vectors & $\unif\bigl(\Sp(k)^2\times\Sp(n-k)\times\Flag(1,\dots,k,\mathbb{H}^{n-k})\bigr)$\\
$\coe(n)$ & Autonne--Takagi vectors & $\unif\bigl(\LGr(\mathbb{R}^{2n})\bigr)$\\
$\cle(n)$ & Autonne--Takagi vectors & $\unif\bigl(\LGr(\mathbb{C}^{2n})\bigr)$\\
$\cse(n)$ & Autonne--Takagi vectors & $\unif\bigl(\LGr(\mathbb{H}^{2n})\bigr)$
\end{tabular}
\end{table}

There is a very well-known differential geometric angle of random matrix theory, namely, the (almost) one-to-one correspondence between symmetric spaces and classical random matrix ensembles \cite{EJ} dating back to Dyson \cite{Dyson}. We emphasize that this bears little relation to our work here. Apart from the Lagrangian Grassmannians, none of the other manifolds in Table~\ref{tab:sum} are symmetric spaces.

\section{Matrices, manifolds, and ensembles}\label{sec:back}

For easy reference, and to set notations and nomenclature, we will go over some essential background materials. While most of these are not new, they are not a mere replica of existing literature and cannot be found in any single source. There is perhaps some value in tabulating them here.

\subsection{Quaternions}

We treat quaternions in the standard way as $x=a+b\mathrm{i}+c\mathrm{j}+d\mathrm{k}$ with $a,b,c,d\in\mathbb{R}$ and $\mathrm{i},\mathrm{j},\mathrm{k}$ denoting the imaginary units. As a reminder, multiplication is given by $(a_1+b_1\mathrm{i}+c_1\mathrm{j}+d_1\mathrm{k})(a_2+b_2\mathrm{i}+c_2\mathrm{j}+d_2\mathrm{k})=(a_1a_2-b_1b_2-c_1c_2-d_1d_2)+(a_1b_2+b_1a_2+c_1d_2-d_1c_2)\mathrm{i}+(a_1c_2-b_1d_2+c_1a_2+d_1b_2)\mathrm{j}+(a_1d_2+b_1c_2-c_1b_2+d_1a_2)\mathrm{k}$; the conjugation of $x=a+b\mathrm{i}+c\mathrm{j}+d\mathrm{k}$ is defined as $\Bar{x}=a-b\mathrm{i}-c\mathrm{j}-d\mathrm{k}$; and modulus as $\lvert x \rvert \coloneqq \sqrt{x \Bar{x}}$. We also regard $\mathbb{R}\subset\mathbb{C}\subset\mathbb{H}$ in the usual way, i.e., as an inclusion of real $*$-subalgebras.

\subsection{Matrices}\label{sec:mat}

We write $\mathbb{R}^{m \times n}$, $\mathbb{C}^{m \times n}$, $\mathbb{H}^{m \times n}$ for the set of $m \times n$ matrices over the respective field/skew field. We denote transpose by $X^\tp$ and conjugate transpose by $X^\h$. Since $\mathbb{R}\subset\mathbb{C}\subset\mathbb{H}$, there is no ambiguity in writing $X^\h$ regardless of whether $X\in\mathbb{C}^{m \times n}$ or $X\in\mathbb{H}^{m \times n}$. While $X^\h = X^\tp$ for $X\in\mathbb{R}^{m \times n}$, we need both notations as we will have to discuss complex symmetric matrices as well.

We denote the \emph{real} vector spaces of real symmetric matrices, complex Hermitian matrices, and quaternionic self-dual matrices respectively as in the left column below:
\begin{equation}\label{eq:basic}
\begin{aligned}
\mathsf{S}^2(\mathbb{R}^n) &=\{X\in\mathbb{R}^{n\times n}: X=X^\tp\},\\
\mathsf{H}^2(\mathbb{C}^n) &=\{X\in\mathbb{C}^{n\times n}: X=X^\h\},\\
\mathsf{Q}^2(\mathbb{H}^n) &=\{X\in \mathbb{H}^{n\times n}:X=X^\h\},
\end{aligned}
\qquad
\begin{aligned}
\O(n) &=\{Q\in\mathbb{R}^{n\times n}: Q^\tp Q = I \},\\
\U(n) &=\{Q\in\mathbb{C}^{n\times n}: Q^\h Q = I\},\\
\Sp(n) &=\{Q\in\mathbb{H}^{n\times n}: Q^\h Q = I\},
\end{aligned}
\end{equation}
and the \emph{real} algebraic groups of orthogonal, unitary, and quaternionic unitary matrices as in the right column above.  Note that here $\Sp(n) \cong \U(2n) \cap \Sp(2n, \mathbb{C})$ is the compact symplectic group as opposed to the noncompact one $\Sp(2n, \mathbb{C})$.

We define two special matrices
\begin{equation}\label{eq:J}
 J_{2n} \coloneqq \begin{bmatrix}
    0&-I_n\\
    I_n &0
\end{bmatrix} \in \mathbb{R}^{2n \times 2n}\quad\text{and} \quad I_{m,n} \coloneqq \begin{bmatrix}
    I_m & 0\\
    0 & -I_n
\end{bmatrix} \in \mathbb{R}^{(m+n) \times (m + n)}.
\end{equation}
For $X \in \mathbb{C}^{2n \times 2n}$, we define its symplectic adjoint and Lagrangian adjoint by
\begin{equation}\label{eq:sa}
X^\s \coloneqq -J_{2n} X^\tp J_{2n} \quad\text{and}\quad X^\la\coloneqq I_{n,n} X^\h I_{n,n}
\end{equation}
respectively.
The vector spaces of complex symmetric matrices, Lagrangian symmetric matrices, and symplectically symmetric matrices are defined by
\begin{equation}\label{eq:basic2}
\begin{aligned}
\mathsf{S}^2(\mathbb{C}^n) &\coloneqq \{X\in\mathbb{C}^{n\times n}: X=X^\tp\}, \\
\mathsf{V}^2(\mathbb{C}^{2n}) &\coloneqq \{X\in\mathbb{C}^{2n\times 2n}: X=X^\la\},\\
\mathsf{Y}^2(\mathbb{C}^{2n}) &\coloneqq \{X\in\mathbb{C}^{2n\times 2n}: X=X^\s\}
\end{aligned}
\end{equation}
respectively. Observe that $\mathsf{S}^2(\mathbb{C}^n)$ and $\mathsf{Y}^2(\mathbb{C}^{2n})$ are complex vector spaces but $\mathsf{V}^2(\mathbb{C}^{2n})$ is only a real vector space.

A delightful fact about quaternionic matrix theory is that the eigenvalue, singular value, and cosine-sine decompositions can be extended almost verbatim to matrices with quaternionic entries \cite[Theorems~1.4.10, 1.4.11, 1.4.12]{QCS}. We present these alongside their well-known real and complex counterparts in Table~\ref{tab:decomp} for easy comparison. 
\begin{table}[h]
\centering
\small
\caption{Matrix decompositions with orderings of eigenvalues $\lambda_1 \ge \dots \ge \lambda_n$, singular values $\sigma_1 \ge \dots \ge \sigma_{\min(m,n)}$, cosine-sine values $c_1 \ge \dots \ge c_p$, $s_1 \le \dots \le s_p$.}
\label{tab:decomp}
\renewcommand{\arraystretch}{1.3}
\begin{tabular}{l@{\hspace{3\tabcolsep}}l@{\hspace{3\tabcolsep}}l@{\hspace{.5\tabcolsep}}l}
\multicolumn{4}{c}{\textsc{eigenvalue decomposition}} \\
$X\in\mathsf{S}^2(\mathbb{R}^n)$ & $X=Q\Lambda Q^\tp$ & $Q\in\O(n)$, & $\Lambda = \diag(\lambda_1,\dots,\lambda_n) \in \mathbb{R}^{n \times n}$ \\
$X\in\mathsf{H}^2(\mathbb{C}^n)$ & $X=Q\Lambda Q^\h$ & $Q\in\U(n)$, & $\Lambda = \diag(\lambda_1,\dots,\lambda_n) \in \mathbb{R}^{n \times n}$ \\
$X\in\mathsf{Q}^2(\mathbb{H}^n)$ & $X=Q\Lambda Q^\h$ & $Q\in\Sp(n)$, &  $\Lambda = \diag(\lambda_1,\dots,\lambda_n) \in \mathbb{R}^{n \times n}$
\end{tabular}

\begin{tabular}{l@{\hspace{3\tabcolsep}}l@{\hspace{3\tabcolsep}}l@{\hspace{.5\tabcolsep}}l@{\hspace{.5\tabcolsep}}l}
\multicolumn{5}{c}{\textsc{singular value decomposition}} \\
$Y\in\mathbb{R}^{m \times n}$ & $Y=U\Lambda V^\tp$ & $U\in\O(m)$, & $V\in\O(n)$, & $\Sigma = \diag(\sigma_1,\dots,\sigma_{\min(m,n)}) \in \mathbb{R}^{m \times n}_\p$ \\
$Y\in\mathbb{C}^{m \times n}$ & $Y=U\Lambda V^\h$ & $U\in\U(m)$, & $V\in\U(n)$, & $\Sigma = \diag(\sigma_1,\dots,\sigma_{\min(m,n)}) \in \mathbb{R}^{m \times n}_\p$ \\
$Y\in\mathbb{H}^{m \times n}$ & $Y=U\Lambda V^\h$ & $U\in\Sp(m)$, & $V\in\Sp(n)$, & $\Sigma = \diag(\sigma_1,\dots,\sigma_{\min(m,n)}) \in \mathbb{R}^{m \times n}_\p$
\end{tabular}

\begin{tabular}{ll>{\raggedright\arraybackslash}m{60ex}}
\multicolumn{3}{c}{\textsc{cosine-sine decomposition}} \\[0.5ex]
$Q\in\O(n)$ &
$Q =
\begin{bsmallmatrix}
    U_1&0\\
    0&U_2
\end{bsmallmatrix}
\begin{bsmallmatrix}
    C&S&0\\
    -S&C&0\\
    0&0&I_{n-2k}
\end{bsmallmatrix}
\begin{bsmallmatrix}
    V_1^\tp&0\\
    0&V_2^\tp
\end{bsmallmatrix}$ &
$U_1, V_1\in \O(k)$; $U_2, V_2 \in \O(n-k)$;  $C^2+S^2=I_k$, 
$C =\diag(c_1,\dots,c_k),\ S =\diag(s_1,\dots,s_k) \in [0,1]^{k \times k}$ \\[1ex]

$Q\in\U(n)$ &
$Q=
\begin{bsmallmatrix}
    U_1&0\\
    0&U_2
\end{bsmallmatrix}
\begin{bsmallmatrix}
    C&S&0\\
    -S&C&0\\
    0&0&I_{n-2k}
\end{bsmallmatrix}
\begin{bsmallmatrix}
    V_1^\h&0\\
    0&V_2^\h
\end{bsmallmatrix}$ &
$U_1, V_1\in \U(k)$; $U_2, V_2 \in \U(n-k)$;  $C^2+S^2=I_k$, 
$C =\diag(c_1,\dots,c_k),\ S =\diag(s_1,\dots,s_k) \in [0,1]^{k \times k}$ \\[1ex]

$Q\in\Sp(n)$ &
$Q =
\begin{bsmallmatrix}
    U_1&0\\
    0&U_2
\end{bsmallmatrix}
\begin{bsmallmatrix}
    C&S&0\\
    -S&C&0\\
    0&0&I_{n-2k}
\end{bsmallmatrix}
\begin{bsmallmatrix}
    V_1^\h&0\\
    0&V_2^\h
\end{bsmallmatrix}$ &
$U_1, V_1\in \Sp(k)$; $U_2, V_2 \in \Sp(n-k)$;  $C^2+S^2=I_k$,
$C =\diag(c_1,\dots,c_k),\ S =\diag(s_1,\dots,s_k) \in [0,1]^{k \times k}$
\end{tabular}

\begin{tabular}{l@{\hspace{3\tabcolsep}}l@{\hspace{3\tabcolsep}}l@{\hspace{.5\tabcolsep}}l}
\multicolumn{4}{c}{\textsc{autonne--takagi decomposition}} \\
$Z\in \mathsf{S}^2(\mathbb{C}^n)$ & $Z=Q \Sigma Q^\tp$ & $Q\in\U(n)$, & $\Sigma = \diag(\sigma_1,\dots,\sigma_n) \in \mathbb{R}^{n \times n}_\p$ \\
$Z\in \mathsf{V}^2(\mathbb{C}^{2n})$ & $Z=Q \Sigma Q^\la$ & $Q\in\U(2n)$, & $\Sigma = \diag(\sigma_1,\dots,\sigma_{2n}) \in \mathbb{R}^{2n \times 2n}$ \\
$Z\in \mathsf{Y}^2(\mathbb{C}^{2n})$ & $Z=Q \Sigma Q^\s$ & $Q\in\U(2n)$, & $\Sigma = \diag(\sigma_1,\dots,\sigma_{2n}) \in \mathbb{R}^{2n \times 2n}_\p$
\end{tabular}
\end{table}

The Autonne--Takagi decomposition is usually stated for matrices in $\mathsf{S}^2(\mathbb{C}^n)$. The existence of such a decomposition for matrices in $\mathsf{Y}^2(\mathbb{C}^{2n})$ and $\mathsf{V}^2(\mathbb{C}^{2n})$ is, as far as we know, new. We will establish these in Section~\ref{sec:circ2}. We emphasize that they are also distinct from the quaternionic Autonne--Takagi decomposition in \cite{QAT} and the QDQ decomposition in \cite[p.~16]{EJ}.

The cosine-sine decomposition is likely less familiar to readers and we will add a few words: Note that it depends on the choice of a $k \in \{1,\dots,\lfloor n/2 \rfloor\}$, usually chosen to be $k = \lfloor n/2 \rfloor$ so that the $I_{n-2k}$ term is either $1$, when $n$ is odd, or nonexistent, when $n$ is even. It is customary to implicitly regard $Q$ as being partitioned into $2\times 2$ blocks of sizes $k \times k$, $k \times (n-k)$, $(n-k) \times k$, and $(n-k) \times (n-k)$:
\[
Q=
\begin{bmatrix}
    Q_{11}&Q_{12}\\
    Q_{21}&Q_{22}
\end{bmatrix}.
\]
We will call $(c_i, s_i) \in [0,1]^2$ a cosine-sine values pair of $Q$; and a pair of column vectors $(u_i, v_i)$ in $(U_1,V_1)$ or $(U_2,V_2)$ left or right cosine-sine vectors pair of $Q$ respectively \cite[Section~2.6]{GV}.

In Table~\ref{tab:decomp}, eigenvectors, singular vectors, and cosine-sine vectors are column vectors of matrices in $\O(n)$, $\U(n)$, or $\Sp(n)$; Autonne--Takagi vectors are column vectors of matrices in $\U(n)$ or $\U(2n)$. Thus they automatically have unit $2$-norms. This will be an implicit assumption throughout our article. However, note that even with this unit norm assumption and the imposed ordering of eigenvalues in Table~\ref{tab:decomp}, an eigenvector of a simple eigenvalue (i.e., of multiplicity one) is still not unique, and is only determined up to scaling by an element of $\O(1)$, $\U(1)$, or $\Sp(1)$, depending on the field.  This scaling indeterminacy is also why one cannot simply regard the eigenvectors distributions as distributions on $\O(n)$, $\U(n)$, or $\Sp(n)$. Singular vector pairs and cosine-sine vector pairs suffer from similar scaling indeterminacies that we will discuss in Sections~\ref{sec:gini} and \ref{sec:circ}. In short, the decompositions in Table~\ref{tab:decomp} should be taken as existential results --- these decompositions exist, but they are never unique.

We will adopt standard orderings of the eigen, singular, and cosine-sine values.
For later purposes, in order to get well-defined maps, we will need to consider dense open subsets of matrices with either distinct eigenvalues, singular values, and cosine-sine values within various sets of matrices. These will be marked with a subscript $\times$. So
\begin{equation}\label{eq:x}
\begin{aligned}
\mathsf{S}^2_\x(\mathbb{R}^n) &=\{X\in\mathsf{S}^2(\mathbb{R}^n): X \text{ has all eigenvalues distinct}\},\\
\mathbb{R}_\x^{m \times n} &=\{Y\in\mathbb{R}^{m \times n}: Y \text{ has all singular values distinct}\},\\
\O_\x(n) &=\{Q\in\O(n): Q \text{ has all cosine-sine values distinct}\},
\end{aligned}
\end{equation}
and similarly for the other sets of matrices appearing in \eqref{eq:basic}. For reasons that will become clear in Section~\ref{sec:circ2}, Autonne--Takagi values are excluded from the discussion in this paragraph.

\subsection{Manifolds}\label{sec:man}

In this article we will identify a manifold with a chosen representation as a submanifold of matrices. For example, the real flag manifold will be identified with its representation as a submanifold in $\mathsf{S}^2(\mathbb{R}^n)$, i.e., replacing ``$\cong$'' with ``$=$''  in \eqref{eq:flag}. We emphasize that the ``$\cong$'' here is not any diffeomorphism but an $\O(n)$-equivariant diffeomorphism, i.e., the natural $\O(n)$-action on the homogeneous space is preserved. We will call such representations \emph{matrix submanifolds} \cite{LK24a}.

\begin{table}[h]
\centering
\footnotesize
\caption{Stiefel, Grassmann, flag, and Lagrangian Grassmann manifolds over $\mathbb{R}$, $\mathbb{C}$, and $\mathbb{H}$. Here $\Delta \coloneqq \diag( I_k, -I_{n-k})$ and  $\Delta_a  \coloneqq \diag(a_1I_{k_1-k_0},\dots,a_{p+1}I_{k_{p+1}-k_p})$ where $a_1,\dots,a_{p+1}$ are any distinct arbitrary real constants.}
\label{tab:rep}
\renewcommand{\arraystretch}{1.3}
\begin{tabular}{l|l|l|l}
\textsc{manifold} & \textsc{abstract} & \textsc{homogeneous} & \textsc{submanifold} \\
\hline
$\V(k, \mathbb{R}^n)$ & $k$-frames in $\mathbb{R}^n$ & $\sfrac{\O(n)}{\O(n-k)}$ & $\{X \in \mathbb{R}^{n \times k}: X^{\tp}X=I\}$ \\
$\Gr(k, \mathbb{R}^n)$  & $k$-planes in $\mathbb{R}^n$ & $\sfrac{\O(n)}{\bigl(\O(k)\times\O(n-k)\bigr)}$ & $\{Q\Delta Q^\tp \in \mathbb{R}^{n \times n} :Q\in\O(n)\}$ \\
$\Flag(k_1,\dots,k_p,\mathbb{R}^n)$ & $(k_1,\dots,k_p)$-flags in $\mathbb{R}^n$ & $\sfrac{\O(n)}{\bigl(\O(k_1)\times\dots\times\O(n-k_p)\bigr)}$ & $\{Q\Delta_aQ^\tp \in \mathbb{R}^{n \times n} :Q\in\O(n)\}$ \\
$\V(k, \mathbb{C}^n)$ & $k$-frames in $\mathbb{C}^n$ & $\sfrac{\U(n)}{\U(n-k)}$ & $\{X \in \mathbb{C}^{n \times k}: X^\h X=I\}$ \\
$\Gr(k, \mathbb{C}^n)$ & $k$-planes in $\mathbb{C}^n$ & $\sfrac{\U(n)}{\bigl(\U(k)\times\U(n-k)\bigr)}$ & $\{Q\Delta Q^\h \in \mathbb{C}^{n \times n}:Q\in\U(n)\}$ \\
$\Flag(k_1,\dots,k_p,\mathbb{C}^n)$ & $(k_1,\dots,k_p)$-flags in $\mathbb{C}^n$ & $\sfrac{\U(n)}{\bigl(\U(k_1)\times\dots\times\U(n-k_p)\bigr)}$ & $\{Q\Delta_aQ^\h \in \mathbb{C}^{n \times n} :Q\in\U(n)\}$  \\
$\V(k, \mathbb{H}^n)$ & $k$-frames in $\mathbb{H}^n$ & $\sfrac{\Sp(n)}{\Sp(n-k)}$ & $\{X \in \mathbb{H}^{n \times k}: X^\h X=I\}$ \\
$\Gr(k, \mathbb{H}^n)$ & $k$-planes in $\mathbb{H}^n$ & $\sfrac{\Sp(n)}{\bigl(\Sp(k)\times\Sp(n-k)\bigr)}$ & $\{Q\Delta Q^\h \in \mathbb{H}^{n \times n}:Q\in\Sp(n)\}$ \\
$\Flag(k_1,\dots,k_p,\mathbb{H}^n)$ & $(k_1,\dots,k_p)$-flags in $\mathbb{H}^n$ & $\sfrac{\Sp(n)}{\bigl(\Sp(k_1)\times\dots\times\Sp(n-k_p)\bigr)}$ & $\{Q\Delta_aQ^\h \in \mathbb{H}^{n \times n} :Q\in\Sp(n)\}$\\
$\LGr(\mathbb{R}^{2n})$ & Lagrangian subspaces in $\mathbb{R}^{2n}$ & $\sfrac{\U(n)}{\O(n)}$ & $\{QQ^\tp\in\mathbb{C}^{n\times n}:Q\in\U(n)\}$\\
$\LGr(\mathbb{C}^{2n})$ & Lagrangian subspaces in $\mathbb{C}^{2n}$ & $\sfrac{\Sp(n)}{\U(n)}$ & $\{QQ^\la\in\mathbb{C}^{2n\times 2n}:Q\in\Sp(n)\}$\\
$\LGr(\mathbb{H}^{2n})$ & Lagrangian subspaces in $\mathbb{H}^{2n}$ & $\sfrac{\U(2n)}{\Sp(n)}$ & $\{QQ^\s\in\mathbb{C}^{2n\times 2n}:Q\in\U(2n)\}$
\end{tabular}
\end{table}

Not every homogeneous space $G/H$ formed out of matrix groups $G$ and $H$ can be represented as a matrix submanifold, i.e., via a $G$-equivariant embedding into a space of matrices. For example, the oriented Grassmannian $\SO(n)/\bigl(\SO(k) \times \SO(n - k)\bigr)$ has no such representation. Nevertheless, the homogeneous spaces of interest to us do admit representations as matrix submanifolds. We present a list in Table~\ref{tab:rep} and henceforth identify the manifold in the first column with its representation as a matrix submanifold in the last column. 

One major theme of Sections~\ref{sec:gauss}--\ref{sec:circ2} is in showing that the parameter spaces of eigenvectors, singular vectors, cosine-sine vectors, or Autonne--Takagi vectors are all matrix submanifolds, i.e., where points are actual matrices, not equivalence classes of matrices. Another major theme is showing that the random matrix ensembles all map to the unique $G$-invariant distributions on these matrix submanifolds, as we have alluded to in Table~\ref{tab:sum}.

While the homogeneous space representations in the third column of Table~\ref{tab:rep} are well known \cite{Helgason01}, the matrix submanifold representations in the fourth column are less so \cite{LK24a}; in fact over $\mathbb{H}$ we are unable to find any reference and we will include their proofs in Sections~\ref{sec:gauss}. The same goes for the complex Lagrangian Grassmannian --- we will derive its matrix submanifold representation in Section~\ref{sec:circ2}.

Note that in the $p = 1$ case, choosing $\Delta \coloneqq I_{k,n-k} = \diag( I_k, -I_{n-k})$ allows us to embed a Grassmannian in a compact Lie group:
\[
\Gr(k, \mathbb{R}^n) \subseteq \O(n), \quad \Gr(k, \mathbb{C}^n) \subseteq \U(n), \quad \Gr(k, \mathbb{H}^n) \subseteq \Sp(n).
\]
This consideration essentially makes $\Delta = I_{k,n-k}$ a canonical choice. However, no such embedding is possible when $p > 1$ and so in this case we allow $a_1,\dots,a_{p+1} \in \mathbb{R}$ to be any arbitrary distinct constants and set
\[
\Delta_a  \coloneqq \diag(a_1I_{k_1-k_0},\dots,a_{p+1}I_{k_{p+1}-k_p})
\]
throughout this article. When discussing a flag manifold $\Flag(k_1,\dots,k_p,\mathbb{R}^n)$, we will write
\begin{equation}\label{eq:flagnot}
k_0\coloneqq 0; \qquad k_{p+1} \coloneqq n; \qquad n_j \coloneqq k_j-k_{j-1}, \; j =1,\dots p.
\end{equation}

As we will be discussing the probability distributions of the manifolds in Table~\ref{tab:rep}, we will need their volumes with respect to their Riemannian volume forms. In the following, $\Gamma$ denotes the Gamma function. We will also define
\[
\theta(r,s)\coloneqq 2^{r/2}\pi^{s/4}, \qquad \gamma(l,m,\beta)\coloneqq\prod_{i=l}^m \Gamma(\beta i/2)
\]
where $r,s \in \mathbb{Z}$, $l,m \in \mathbb{N}$, and $\beta \in \{1,2,4\}$. The first row of the following table may be found in \cite[Equations~6.5, 5.16, 6.17]{Volume1}, from which we derived the subsequent rows using the homogeneous space characterizations in Table~\ref{tab:rep}.
\begin{table}[h]
\centering
\small
\caption{Volumes of manifolds over $\mathbb{F} = \mathbb{R}$, $\mathbb{C}$, $\mathbb{H}$.}
\label{tab:vol}
\vspace*{-2ex}
\renewcommand{\arraystretch}{1.3}
\begin{tabular}{l|c|c|c}
\multirow{2}{*}{\textsc{manifold}} & \multicolumn{3}{c}{\textsc{volume}} \\
& $\mathbb{R}$ & $\mathbb{C}$ & $\mathbb{H}$ \\
\hline
&&&\\[-2ex]
$\O(n)/\U(n)/\Sp(n)$ &$\frac{\theta(2n,n(n+1))}{\gamma(1,n,1)}$ &$\frac{\sqrt{n}\theta(n+1,2n(n+1))}{\gamma(1,n,2)}$ & $\frac{\theta(2n,4n(n+1))}{\gamma(1,n,4)}$\\[2ex]
$\V(k,\mathbb{F}^n)$   & $\frac{\theta(2k,k(2n-k+1))}{\gamma(n-k+1,n,1)}$&$\sqrt{\frac{n}{n-k}}\frac{\theta(k,2k(2n-k+1))}{\gamma(n-k+1,n,2)}$ & $\frac{\theta(2k,4k(2n-k+1))}{\gamma(n-k+1,n,4)}$\\[2ex]
$\Gr(k,\mathbb{F}^n)$     & $\frac{\theta(0,2k(n-k))\gamma(1,k,1)}{\gamma(n-k+1,n,1)}$&$\sqrt{\frac{n}{k(n-k)}}\frac{\theta(-1,4k(n-k))\gamma(1,k,2)}{\gamma(n-k+1,n,2)}$ & $\frac{\theta(0,8k(n-k))\gamma(1,k,4)}{\gamma(n-k+1,n,4)}$\\[2ex]
$\Flag(k_1,\dots,k_p,\mathbb{F}^n)$        &$\frac{\theta(0,n^2)\prod_{j=1}^{p+1}\gamma(1,n_j,1)}{\theta(0,\sum_{j=1}^{p+1}n_j^2)\gamma(1,n,1)}$ &$\sqrt{\frac{n}{\prod_{j=1}^{p+1}n_j}}\frac{\theta(1-n,2n^2)\prod_{j=1}^{p+1}\gamma(1,n_j,2)}{\theta(0,2\sum_{j=1}^{p+1}n_j^2)\gamma(1,n,2)}$ &$\frac{\theta(0,4n^2)\prod_{j=1}^{p+1}\gamma(1,n_j,4)}{\theta(0,4\sum_{j=1}^{p+1}n_j^2)\gamma(1,n,4)}$\\[2ex]
$\Flag(\mathbb{F}^n)$    & $\frac{\theta(0,n(n+1))}{\gamma(1,n,1)}$&$\frac{\sqrt{n}\theta(1-n,2n(n-1))}{\gamma(1,n,2)}$ & $\frac{\theta(0,4n(n-1))}{\gamma(1,n,4)}$\\[2ex]
$\LGr(\mathbb{F}^{2n})$ & $\frac{\sqrt{n}\theta(1-n,n(n+1))\gamma(1,n,1)}{\gamma(1,n,2)}$ &$\frac{\theta(n-1,2n(n+1))\gamma(1,n,2)}{\sqrt{n}\gamma(1,n,4)}$ & $\frac{\sqrt{2n}\theta(1,4n^2)\gamma(1,n,4)}{\gamma(1,2n,2)}$
\end{tabular}
\end{table}

\subsection{Ensembles}\label{sec:dis}

In this article, the term ``distributions'' refers to probability distributions and the term ``ensembles'' exclusively refers to distributions on the vector spaces or groups of matrices appearing in Section~\ref{sec:mat}. Specifically,
\begin{enumerate}[{label=\upshape(\alph*)}]
\item\label{it:orth} $\goe(n)$, $\loe(m,n)$, $\joe(l,m,n)$ are distributions on $\mathsf{S}^2(\mathbb{R}^n)$;
\item $\gue(n)$, $\lue(m,n)$, $\jue(l,m,n)$ are distributions on $\mathsf{H}^2(\mathbb{C}^n)$;
\item $\gse(n)$, $\lse(m,n)$, $\jse(l,m,n)$ are distributions on $\mathsf{Q}^2(\mathbb{H}^n)$;
\item\label{it:gini} $\gio(m,n)$, $\giu(m,n)$, $\gis(m,n)$ are distributions on $\mathbb{R}^{m \times n}$, $\mathbb{C}^{m \times n}$, $\mathbb{H}^{m \times n}$;
\item\label{it:circ} $\cre(n)$, $\cue(n)$, $\cqe(n)$ are distributions on $\O(n)$, $\U(n)$, $\Sp(n)$;
\item \label{it:circ2} $\coe(n)$, $\cle(n)$, $\cse(n)$ are distributions on $\LGr(\mathbb{R}^{2n})$, $\LGr(\mathbb{C}^{2n})$, $\LGr(\mathbb{H}^{2n})$.
\end{enumerate}
GOE/GUE/GSE are collectively called the Gaussian ensembles, or sometimes Hermite ensembles \cite[p.~1127]{Hermite}. LOE/LUE/LSE are collectively called the Laguerre ensembles, or sometimes Wishart ensembles \cite[Definition~3.1.3]{Wishart}. JOE/JUE/JSE are collectively called the Jacobi ensembles, or sometimes multivariate beta ensembles \cite[Definition~3.3.2]{Wishart}. The distributions in \ref{it:gini} are called Ginibre ensembles \cite[p.~752]{simpspec}.  Those in \ref{it:circ} and \ref{it:circ2} are called circular ensembles \cite[p.~1057]{Cir} in the literature, except for CLE, which we are proposing as a hitherto missing candidate for the trio in \ref{it:circ2}, as we will see in \eqref{eq:include}.

The ensembles on vector spaces of matrices in \ref{it:orth}--\ref{it:gini} have probability density functions as shown in Table~\ref{tab:pdf}. For Gaussian and Ginibre ensembles, they are routine easy calculations. For Laguerre and Jacobi ensembles, they may be found in \cite[Proposition~3.2.7]{simpspec} and \cite[Theorem~3.3.1]{Wishart} respectively, where the $\beta$-multivariate Gamma function \cite[Theorem~2.1.11]{Wishart} is given by
\[
\Gamma_n^{\beta}(s) \coloneqq \pi^{\beta n(n-1)/4}\prod_{j=1}^n\Gamma\bigl(s-(j-1)\beta/2\bigr)
\]
for $\beta\in\{1,2,4\}$ and any positive half-integer $s > (n-1)/2$. 

\begin{table}[h]
\centering
\small
\caption{$\beta = 1,2,4$ corresponds to $\mathbb{R}$, $\mathbb{C}$, $\mathbb{H}$ respectively. Note that $Y^\h = Y^\tp$ over $\mathbb{R}$.}
\label{tab:pdf}
\renewcommand{\arraystretch}{1.3}
\begin{tabular}{l|l}
\textsc{ensemble} & \textsc{probability density} \\
\hline
&\\[-2ex]
$\goe$/$\gue$/$\gse(n)$
 & $2^{-n/2}\pi^{-\beta n(n-1 +2/\beta)/4}e^{-\tr(X^2)/2}$ \\[2ex]
$\loe$/$\lue$/$\lse(m,n)$
 & $\frac{1}{(2/\beta)^{\beta mn/2}\Gamma_n^{\beta}(\beta m/2)}(\det X)^{(m-n+1)\beta/2-1}e^{-\tr(X)\beta/2}$\\[2ex]
$\joe$/$\jue$/$\jse(l,m,n)$
 & $\frac{\Gamma_n^{\beta}((l+m)\beta/2)}
         {\Gamma_n^{\beta}(l\beta/2)
          \Gamma_n^{\beta}(m\beta/2)}
    (\det X)^{(l-n+1)\beta/2-1}
    \det(I_n-X)^{(m-n+1)\beta/2-1}$ \\[2ex]
$\gio$/$\giu$/$\gis(m,n)$
 & $(\beta/2\pi)^{mn\beta/2}e^{-\tr(Y^\h Y)\beta/2}$
\end{tabular}
\end{table}

Unlike the ensembles \ref{it:orth}--\ref{it:gini} for vector spaces of matrices, the circular ensembles in \ref{it:circ} and \ref{it:circ2}  are distributions on groups and homogeneous spaces of matrices respectively. In general, any locally compact group $G$ has a left-invariant positive measure unique up to scaling by a positive constant. Any homogeneous space $G/H$ where $H$ is a closed subgroup of $G$ has a positive $G$-invariant measure on $G/H$ if and only if the modular function of $G$, when restricted to $H$, equals the modular function of $H$ \cite[Theorem~2.51]{Folland1}. This measure is unique up to scaling by a positive constant. For a compact $G$, $H$ is also compact, and thus both are unimodular. Hence we have a unique probability measure on $G/H$ that is $G$-invariant \cite[Proposition~2.27]{Folland1}. We call this unique probability measure the \emph{uniform distribution} and denote it by $\unif(G/H)$.  We state this as a proposition for easy reference later.
\begin{proposition}\label{prop:unif}
Let $H$ be a closed subgroup of a compact group $G$. Then $\unif(G/H)$ is the unique distribution on the homogeneous space $G/H$ that is $G$-invariant.
\end{proposition}
The circular ensembles are just alternative names for the following uniform distributions:
\begin{equation}\label{eq:circ}\small
\begin{aligned}
\cre(n) &= \unif\bigl(\O(n)\bigr),  &\cue(n) &= \unif \bigl( \U(n) \bigr), &\cqe(n) &= \unif \bigl( \Sp(n) \bigr), \\
\coe(n) &= \unif\bigl(\U(n)/\O(n)\bigr), &\cle(n) &= \unif\bigl(\Sp(n)/\U(n)\bigr), &\cse(n) &= \unif \bigl( \U(2n)/\Sp(n) \bigr).
\end{aligned}
\end{equation}
More generally, for any of the manifolds appearing in Table~\ref{tab:rep}, its uniform distribution is given by that of its homogeneous space representation (in the third column). The probability densities for all these uniform distributions are given by the constant function with their respective volume in Table~\ref{tab:vol} as normalization constant.

Although we will rely solely on the probability densities in Table~\ref{tab:pdf} for our results in this article, there is a more common entrywise  description for all but three of these ensembles. We include it in Appendix~\ref{app:alt} for completeness. There are authors (e.g., \cite[p.~111]{simpspec}) who use the term ``ensembles'' to refer to the distributions of eigenvalues, singular values, or cosine-sine values. We will refrain from such ambiguity. In fact, these distributions will not appear in our article as we are exclusively interested in the distributions of eigenvectors, singular vectors, cosine-sine vectors, and Autonne--Takagi vectors of random matrix ensembles.

\section{Eigenvectors distributions of Gaussian ensembles}\label{sec:gauss}

The homogeneous space representations of the manifolds in the third column of Table~\ref{tab:rep} are all standard and well-known. The matrix submanifold representations in the fourth column are either well-known or have appeared in \cite{HM94, LK24a} with the quaternionic cases as the only exception. We will provide a proof for the Stiefel, Grassmann, and flag manifolds here, deferring the Lagrangian Grassmannian to Section~\ref{sec:circ2}.
\begin{proposition}[Quaternionic Stiefel, Grassmann, flag manifolds]\label{sti}
We have $\Sp(n)$-equivariant diffeomorphisms:
\begin{align*}
\V(k,\mathbb{H}^n) &=\{X \in \mathbb{H}^{n \times k}: X^\h X=I\} \cong\Sp(n)/\Sp(n-k), \\
\Gr(k,\mathbb{H}^n) &= \{Q\Delta Q^\h \in \mathsf{Q}^2 (\mathbb{H}^n) :Q\in\Sp(n)\}  \cong\Sp(n)/\bigl(\Sp(k) \times \Sp(n-k)\bigr), \\
\Flag(k_1,\dots,k_p,\mathbb{H}^n) &= \{Q\Delta_aQ^\h \in \mathsf{Q}^2 (\mathbb{H}^n) :Q\in\Sp(n)\} \cong\Sp(n)/\bigl(\Sp(k_1)\times\dots\times\Sp(n-k_p)\bigr).
\end{align*}
Note that we identify these manifold with their matrix representations.
\end{proposition}
\begin{proof}
The unitary symplectic group $\Sp(n)$ acts transitively on $\Flag(k_1,\dots,k_p,\mathbb{H}^n)$ by conjugation. The stabilizer of $\Delta_a$ is its centralizer
\[
Z_{\Sp(n)}(\Delta_a)=\Sp(k_1)\times\dots\times\Sp(n-k_p).
\]
Hence $\Flag(k_1,\dots,k_p,\mathbb{H}^n)\cong\Sp(n)/\bigl(\Sp(k_1)\times\dots\times\Sp(n-k_p)\bigr)$. This is an $\Sp(n)$-equivariant diffeomorphism by construction. The Grassmannian is the $p =1$ case.
It is also true that $\Sp(n)$ acts on $\V(k,\mathbb{H}^n)$ by left multiplication. Noticing that for any $Q\in\Sp(n)$, $Q\begin{bsmallmatrix}
    I_k\\
    0
\end{bsmallmatrix}=Q_k$, where $Q_k$ denotes the first $k$ columns of $Q$, this action is transitive, and the stabilizer of $\begin{bsmallmatrix}
    I_k\\
    0
\end{bsmallmatrix}$ is\[
Z_{\Sp(n)}(\begin{bsmallmatrix}
    I_k\\
    0
\end{bsmallmatrix})=\{\diag(I_k,W):W\in\Sp(n-k)\}\cong\Sp(n-k).
\]
Hence $\V(k,\mathbb{H}^n)\cong\Sp(n)/\Sp(n-k)$, and this is an $\Sp(n)$-equivalent diffeomorphism.
\end{proof}

From Table~\ref{tab:pdf}, the probability densities of the Gaussian ensembles are
\[
f_{\beta,n}^\G(X)=\begin{cases}
  2^{-\frac{n}{2}}\pi^{-\frac{n(n+1)}{4}}e^{-\frac{1}{2}\tr(X^2)}&\text{if }\beta=1,\\[0.5ex]
  2^{-\frac{n}{2}}\pi^{-\frac{n^2}{2}}e^{-\frac{1}{2}\tr(X^2)}&\text{if }\beta=2,\\[1ex]
  2^{-\frac{n}{2}}\pi^{-n(n-\frac{1}{2})}e^{-\frac{1}{2}\tr(X^2)}&\text{if }\beta=4.
\end{cases}
\]
So $\goe(n)$ is $\O(n)$-invariant, $\gue(n)$ is $\U(n)$-invariant, and $\gse(n)$ is $\Sp(n)$-invariant. A key to obtaining our result is that the following maps, while not invariant, are equivariant under the adjoint action of  $\O(n)$, $\U(n)$, and $\Sp(n)$.

\begin{definition}\label{def:phi}
Let $0<k_1<\dots<k_p<n$ be integers. We define
\begin{alignat*}{2}
\varphi_{k_1,\dots,k_p} :\mathsf{S}^2_\x (\mathbb{R}^n) &\to\Flag(k_1,\dots,k_p,\mathbb{R}^n), \quad & X &\mapsto Q\Delta_aQ^\tp, \\
\varphi_{k_1,\dots,k_p}:\mathsf{H}^2_\x (\mathbb{C}^n)&\to\Flag(k_1,\dots,k_p,\mathbb{C}^n), \quad &X &\mapsto Q\Delta_aQ^\h, \\
\varphi_{k_1,\dots,k_p}:\mathsf{Q}^2_\x (\mathbb{H}^n)&\to\Flag(k_1,\dots,k_p,\mathbb{H}^n), \quad &X &\mapsto Q\Delta_aQ^\h,
\end{alignat*}
where $X=Q\Lambda Q^\h$ is an eigenvalue decomposition with $\lambda_1 >\dots>\lambda_n $ and $Q\in \Sp(n)$. 
\end{definition}

While the matrix of eigenvectors $Q$ in Definition~\ref{def:phi} is not uniquely defined, we will see in Theorem~\ref{thmfk} that $\varphi_{k_1,\dots,k_p}$ is nevertheless a well-defined function.  Furthermore, $\varphi_{k_1,\dots,k_p} (X)$ is also well-defined as a random matrix as the eigenvalues of a random matrix $X$ in any of these Gaussian ensembles is simple \cite[Proposition~1.3.4]{simpspec}. Formally, using $\goe(n)$ for illustration, $X$ is a matrix-valued random variable, i.e., a measurable function $X : \Omega \to \mathsf{S}^2(\mathbb{R}^n)$, and $\varphi_{k_1,\dots,k_p}$ is defined on a Zariski open subset of $\mathsf{S}^2 (\mathbb{R}^n)$. So the random matrix $\varphi_{k_1,\dots,k_p} (X)$ is the measurable function $\varphi_{k_1,\dots,k_p} \circ  X : \Omega \to \Flag(k_1,\dots,k_p,\mathbb{R}^n)$. Nevertheless such pedantry appears to be uncommon and so we will not repeat this below.

We have slightly abused notations and used $\varphi_{k_1,\dots,k_p} $ to denote three different functions in Definition~\ref{def:phi}. However, their definitions are consistent as $\mathsf{S}^2 (\mathbb{R}^n) \subseteq \mathsf{H}^2 (\mathbb{C}^n) \subseteq \mathsf{Q}^2 (\mathbb{H}^n)$ and
\[
\Flag(k_1,\dots,k_p,\mathbb{R}^n) \subseteq \Flag(k_1,\dots,k_p,\mathbb{C}^n) \subseteq \Flag(k_1,\dots,k_p,\mathbb{H}^n).
\]
So in Definition~\ref{def:phi} the first function is essentially a restriction of the second, which in turn is essentially a restriction of the third. This observation will be used in our proofs too: We will state a result over $\mathbb{R}$, $\mathbb{C}$, and $\mathbb{H}$, but will only prove the quaternionic case, as the real and complex cases may essentially be obtained by restriction. The observant reader will soon notice that the noncommutativity of $\mathbb{H}$ does not require any  special handling in our proofs. The reason is that we are working with matrices, which are already noncommutative, and the algebra of matrices over $\mathbb{H}$ essentially behaves no differently from that over $\mathbb{R}$ and $\mathbb{C}$ \cite{QCS}.

The considerations in the last two paragraphs will apply to the rest of our article in the context of other ensembles and other decompositions. In any proofs that involve both random matrices and constant matrices, we use $A$ and $B$ to denote the latter for easy distinction; this will also apply to the rest of our article.

\begin{theorem}[Gaussian ensembles and flag manifolds]\label{thmfk} Let $0<k_1<\dots<k_p<n$ be integers.
    \begin{enumerate}[{label=\upshape(\roman*)}]
        \item\label{it:goe} If $X \sim \goe(n)$, then $\varphi_{k_1,\dots,k_p}(X) \sim\unif\bigl(\Flag(k_1,\dots,k_p,\mathbb{R}^n)\bigr)$.
        \item\label{it:gue} If $X \sim \gue(n)$, then $\varphi_{k_1,\dots,k_p}(X) \sim\unif\bigl(\Flag(k_1,\dots,k_p,\mathbb{C}^n)\bigr)$.
        \item\label{it:gse} If $X \sim \gse(n)$, then $\varphi_{k_1,\dots,k_p}(X) \sim\unif\bigl(\Flag(k_1,\dots,k_p,\mathbb{H}^n)\bigr)$.
    \end{enumerate}
\end{theorem}
\begin{proof}
We identify $\Sp(1)^n$ with the subgroup of $\Sp(n)$ comprising diagonal matrices whose diagonal entries are unit quaternions. Let $X\sim\gse(n)$ and the subsequent notations be as in Definition~\ref{def:phi}. The random matrix of eigenvectors $Q$ takes value in $\Sp(n)$ and is uniquely defined up to right multiplication by $D \in \Sp(1)^n$ almost surely. Since $QD\Delta_a (QD)^\h=Q\Delta_a Q^\h$, $\varphi_{k_1,\dots,k_p}$ is well-defined. For any $A\in\Sp(n)$, since $AQ\Lambda Q^\h A^\h$ is an eigenvalue decomposition of $AXA^\h$,
\[
\varphi_{k_1,\dots,k_p}(AXA^\h) = AQ\Delta_a Q^\h A^\h = A\varphi_{k_1,\dots,k_p}(X)A^\h.
\]
Since $\gse(n)$ is $\Sp(n)$-invariant, $X$ and $AXA^\h$ have the same distribution, so do $\varphi_{k_1,\dots,k_p}(X)$ and $\varphi_{k_1,\dots,k_p}(AXA^\h)$. Hence $\varphi_{k_1,\dots,k_p}(X)$ and $A \varphi_{k_1,\dots,k_p}(X)A^\h$ are identically distributed for all $A \in \Sp(n)$. By Proposition~\ref{prop:unif}, the uniform distribution is the unique probability distribution on $\Flag(k_1,\dots,k_p,\mathbb{H}^n)$ that is $\Sp(n)$-invariant; we must have $\varphi_{k_1,\dots,k_p}(X) \sim\unif\bigl(\Flag(k_1,\dots,k_p,\mathbb{H}^n)\bigr)$. The other two cases follow by restriction as discussed previously.
\end{proof}

A word about the interpretation of Theorem~\ref{thmfk} is in order. It says that any increasing sequence of eigenspaces of $X \sim \goe(n)$,
\[
\spn\{q_1,\dots,q_{k_1} \} \subseteq \spn\{q_1,\dots,q_{k_2} \} \subseteq \dots \subseteq \spn\{q_1,\dots,q_{k_p} \}
\]
is uniformly distributed in the space of all $(k_1,\dots,k_p)$-flags. Here $q_1,\dots,q_n$ are the column vectors of $Q$, i.e., the eigenvectors of $X$.

The most informative special case is when $p = n-1$, which yields the distribution of a complete set of eigenvectors of these random matrices. Note that in this case $\Delta_a  = \diag(a_1,\dots,a_n)$ is a diagonal matrix with $n$ distinct diagonal entries $a_1,\dots,a_n \in \mathbb{R}$. While the following corollary follows from Theorem~\ref{thmfk}, we include a proof to explain why eigenvalue-ordered eigenbases are naturally parameterized by the complete flag manifold.

\begin{corollary}[Eigenvectors distributions of Gaussian ensembles]\leavevmode
\begin{enumerate}[{label=\upshape(\roman*)}]
    \item If $X \sim \goe(n)$, then $\varphi_{1,2,\dots,n-1}(X)\sim\unif\bigl(\Flag(\mathbb{R}^n)\bigr)$.
    \item If $X \sim \gue(n)$, then $\varphi_{1,2,\dots,n-1}(X)\sim\unif\bigl(\Flag(\mathbb{C}^n)\bigr)$.
    \item If $X \sim \gse(n)$, then $\varphi_{1,2,\dots,n-1}(X)\sim\unif\bigl(\Flag(\mathbb{H}^n)\bigr)$.
\end{enumerate}
\end{corollary}
\begin{proof}
As in the proof of Theorem~\ref{thmfk}, any two eigenvectors of $X \sim \gse(n)$ for the same (simple) eigenvalue differ by an element of $\Sp(1)$. It follows that the space of eigenbases is $\Sp(n)/\Sp(1)^n$. By Proposition~\ref{sti}, $\Sp(n)/\Sp(1)^n \cong \Flag(\mathbb{H}^n)$ with the isomorphism given precisely by $\lb Q \rb \mapsto Q \Delta_a Q^\h$, where $\lb Q \rb$ denote equivalence class.
\end{proof}

The $p = 1$ special case of  Theorem~\ref{thmfk} deserves special highlight too. In this case, we view the result as a method to generate uniformly distributed points on the Grassmannian, arguably the most important flag manifold, using Gaussian ensembles.
\begin{corollary}[Uniform sampling on Grassmannian]\label{cor:Gr}
Let $0 < k < n$ be integers.
\begin{enumerate}[{label=\upshape(\roman*)}]
    \item If $X \sim \goe(n)$, then $\varphi_k(X)\sim\unif\bigl(\Gr(k,\mathbb{R}^n)\bigr)$.
    \item If $X \sim \gue(n)$, then $\varphi_k(X)\sim\unif\bigl(\Gr(k,\mathbb{C}^n)\bigr)$.
    \item If $X \sim \gse(n)$, then $\varphi_k(X)\sim\unif\bigl(\Gr(k,\mathbb{H}^n)\bigr)$.
\end{enumerate}
\end{corollary}

Observant readers might also have noticed that our proof of Theorem~\ref{thmfk}\ref{it:gse} primarily relies only on the $\Sp(n)$-invariance of the probability density function for $\gse(n)$. Clearly, we could also have used the $\O(n)$-invariance of $\goe(n)$ and the $\U(n)$-invariance of $\gue(n)$ to deduce Theorem~\ref{thmfk}\ref{it:goe} and \ref{it:gue}. This observation shows that the conclusion of Theorem~\ref{thmfk} would in fact apply to any random matrix ensemble with $G$-invariant densities. From Table~\ref{tab:pdf},  the probability densities of the Laguerre ensembles are
\[
f_{\beta,m,n}^\La(X) =\begin{cases}
  \frac{1}{2^{mn/2}\Gamma_n^1(m/2)}(\det X)^{(m-n+1)/2-1}e^{-\tr(X)/2} &\text{if }\beta=1,\\[1ex]
  \frac{1}{\Gamma_n^2(m)}(\det X)^{m-n}e^{-\tr(X)} &\text{if }\beta=2,\\[1ex]
  \frac{1}{(1/2)^{2mn}\Gamma_n^4(2m)}(\det X)^{2m-2n+1}e^{-2\tr(X)} &\text{if }\beta=4,
\end{cases}
\]
and those of the Jacobi ensembles are
\[
f_{\beta,l,m,n}^\J(X) =\begin{cases}
  \frac{\Gamma_n^1((l+m)/2)}
         {\Gamma_n^1(l/2)
          \Gamma_n^1(m/2)}
    (\det X)^{(l-n+1)/2-1}
    \det(I_n-X)^{(m-n+1)/2-1} &\text{if }\beta=1,\\[2ex]
  \frac{\Gamma_n^2(l+m)}
         {\Gamma_n^2(l)
          \Gamma_n^2(m)}
    (\det X)^{l-n}
    \det(I_n-X)^{m-n} &\text{if }\beta=2,\\[2ex]
  \frac{\Gamma_n^4(2l+2m)}
         {\Gamma_n^4(2l)
          \Gamma_n^4(2m)}
    (\det X)^{2l-2n+1}
    \det(I_n-X)^{2m-2n+1} &\text{if }\beta=4.
\end{cases}
\]
Clearly the $\beta =1,2,4$ cases are each invariant under conjugation by $\O(n)$, $\U(n)$, $\Sp(n)$ respectively. Hence we deduce the following:
\begin{corollary}[Eigenvectors distributions of Laguerre and Jacobi ensembles]
Let $l \ge n$ and $m \ge n$ be positive integers. Let $\varphi_{k_1,\dots,k_p}$ be as in Definition~\ref{def:phi}. Then we have the following:
\begin{enumerate}[{label=\upshape(\roman*)}]
\item If $X\sim\loe(m,n)$ or $\joe(l,m,n)$, then $\varphi_{k_1,\dots,k_p}(X)\sim\unif\bigl(\Flag(k_1,\dots,k_p,\mathbb{R}^n)\bigr)$.
\item If $X\sim\lue(m,n)$ or $\jue(l,m,n)$, then $\varphi_{k_1,\dots,k_p}(X)\sim\unif\bigl(\Flag(k_1,\dots,k_p,\mathbb{C}^n)\bigr)$. 
\item If $X\sim\lse(m,n)$ or $\jse(l,m,n)$, then $\varphi_{k_1,\dots,k_p}(X)\sim\unif\bigl(\Flag(k_1,\dots,k_p,\mathbb{H}^n)\bigr)$.
\end{enumerate}
\end{corollary}
This shows that the Gaussian, Laguerre, and Jacobi ensembles differ only in their eigenvalues distributions, but have identical eigenvectors distributions.

\section{Singular vectors distributions of Ginibre ensembles}\label{sec:gini}

As we saw in the last section, the key to the results therein is to guess the ``correct'' manifold that parameterizes the eigenvectors. Once we have this, an appropriately chosen map would take the Gaussian ensemble to the uniform distribution on this manifold. For eigenvectors, the complete flag manifold is by and large the obvious candidate; furthermore, as we saw in \eqref{eq:flag} and the last column of Table~\ref{tab:rep}, it can be readily realized as a matrix submanifold. In this section, we will see that the ``correct'' matrix submanifold that parameterizes left and right singular vector pairs would not be that obvious.

From Table~\ref{tab:pdf}, the densities of Ginibre ensembles are given by
\[
f_{\beta,m,n}^\G(Y) =\begin{cases}
  (1/2\pi)^{mn/2}e^{-\tr(Y^\tp Y)/2} &\text{if }\beta=1,\\[1ex]
  (1/\pi)^{mn}e^{-\tr(Y^\h Y)} &\text{if }\beta=2,\\[1ex]
  (2/\pi)^{2mn}e^{-2\tr(Y^\h Y)} &\text{if }\beta=4.
\end{cases}
\]
For our purpose, the only point to note is that $\gio(m,n)$ is $\O(m)\times\O(n)$-invariant, $\giu(m,n)$ is $\U(m)\times\U(n)$-invariant, and $\gis(m,n)$ is $\Sp(m)\times\Sp(n)$-invariant.

Without loss of generality, we will assume that $m \ge n$ throughout this section; the $m < n$ case just follows from applying the results to $Y^\tp$ or $Y^\h$. So any $Y\sim\gis(m,n)$ has $n$ district singular values probability one and we arrange them in descending order as in Table~\ref{tab:decomp}.

\begin{proposition}[Singular vectors distributions of Ginibre ensembles I]\label{prop:sing}
Let  $0 < n \le m$ be integers and
\begin{alignat*}{2}
\psi:\mathbb{R}_\x^{m \times n} &\to [\V(n,\mathbb{R}^m)\times\O(n)]/\O(1)^n, \quad &Y &\mapsto\lb(U,V)\rb_{\O}, \\
\psi:\mathbb{C}_\x ^{m \times n} &\to [\V(n,\mathbb{C}^m)\times\U(n)]/\U(1)^n, \quad &Y &\mapsto\lb(U,V)\rb_{\U}, \\
\psi:\mathbb{H}_\x ^{m \times n} &\to [\V(n,\mathbb{H}^m)\times\Sp(n)]/\Sp(1)^n, \quad &Y &\mapsto\lb(U,V)\rb_{\Sp},
\end{alignat*}
where the equivalence classes of left and right singular vectors are given by
\begin{alignat*}{2}
\lb(U,V)\rb_{\O} &=\lb(UD,VD)\rb_{\O} \quad &\text{for any } D &\in \O(1)^n,\\
\lb(U,V)\rb_{\U}&=\lb(UD,VD)\rb_{\U} \quad &\text{for any } D &\in \U(1)^n,\\
\lb(U,V)\rb_{\Sp}&=\lb(UD,VD)\rb_{\Sp} \quad &\text{for any } D &\in \Sp(1)^n.
\end{alignat*}
Then we have the following:
\begin{enumerate}[{label=\upshape(\roman*)}]
    \item If $Y\sim\gio(m,n)$, then $\psi(Y)\sim\unif\bigl([\V(n,\mathbb{R}^m)\times\O(n)]/\O(1)^n\bigr)$.
    \item If $Y\sim\giu(m,n)$, then $\psi(Y)\sim\unif\bigl([\V(n,\mathbb{C}^m)\times\U(n)]/\U(1)^n\bigr)$.
    \item If $Y\sim\gis(m,n)$, then $\psi(Y)\sim\unif\bigl([\V(n,\mathbb{H}^m)\times\Sp(n)]/\Sp(1)^n\bigr)$.
\end{enumerate}
\end{proposition}
\begin{proof}
    For any $(A, B) \in\Sp(m) \times \Sp(n)$, $Y$ and $AYB^\h$ are identically distributed. So the equivalence classes $\lb(U,V)\rb_{\Sp}$ and $\lb(AU,BV)\rb_{\Sp}$ are also identically distributed. Since
\[
[\V(n,\mathbb{H}^m)\times\Sp(n)]/\Sp(1)^n \cong [\Sp(m)\times\Sp(n)] / [\Sp(m-n) \times \Sp(1)^n],
\]
it follows from Proposition~\ref{prop:unif} that $\psi(Y)\sim\unif\bigl([\V(n,\mathbb{H}^m)\times\Sp(n)]/\Sp(1)^n\bigr)$. The $\mathbb{R}$ and $\mathbb{C}$  cases are specializations of the result for $\mathbb{H}$.
\end{proof}
There is in fact another isomorphism,
\begin{equation}\label{eq:iso1}
[\V(n,\mathbb{H}^m)\times\Sp(n)]/\Sp(1)^n\cong\V(n,\mathbb{H}^m)\times\Flag(\mathbb{H}^n),
\end{equation}
embedding the space of (equivalence classes of) left and right singular vectors as a matrix submanifold in $\mathbb{H}^{n \times (k + n)}$, and similarly over $\mathbb{R}$ and $\mathbb{C}$. This characterization allows us to avoid quotient spaces and equivalence classes altogether. We regard the following as our main result of this section.
\begin{theorem}[Singular vectors distributions of Ginibre ensembles II]\label{thm:gini}
Let  $0 < n \le m$ be integers and
\begin{alignat*}{2}
\xi:\mathbb{R}_\x ^{m \times n}&\to\V(n,\mathbb{R}^m)\times\Flag(\mathbb{R}^n), \quad &Y&\mapsto(UV^\tp,V\Delta_aV^\tp), \\
\xi:\mathbb{C}_\x ^{m \times n}&\to\V(n,\mathbb{C}^m)\times\Flag(\mathbb{C}^n), \quad &Y&\mapsto(UV^\h,V\Delta_aV^\h), \\
\xi:\mathbb{H}_\x ^{m \times n}&\to\V(n,\mathbb{H}^m)\times\Flag(\mathbb{H}^n), \quad &Y&\mapsto(UV^\h,V\Delta_aV^\h).
\end{alignat*}
Then we have the following:
\begin{enumerate}[{label=\upshape(\roman*)}]
    \item If $Y\sim\gio(m,n)$, then $\xi(Y)\sim\unif\bigl(\V(n,\mathbb{R}^m)\times\Flag(\mathbb{R}^n)\bigr)$.
    \item If $Y\sim\giu(m,n)$, then $\xi(Y)\sim\unif\bigl(\V(n,\mathbb{C}^m)\times\Flag(\mathbb{C}^n)\bigr)$.
    \item If $Y\sim\gis(m,n)$, then $\xi(Y)\sim\unif\bigl(\V(n,\mathbb{H}^m)\times\Flag(\mathbb{H}^n)\bigr)$.
    \end{enumerate}
\end{theorem}
\begin{proof}
With probability one, the singular value decomposition is unique modulo the equivalence relation in Proposition~\ref{prop:sing}, i.e., $\lb(U,V)\rb_{\Sp}=\lb(UD,VD)\rb_{\Sp}$ for any $D\in\Sp(1)^n$. Since $UDD^\h V^\h=UV^\h$ and $V\Delta_aV^\h=VD\Delta_{a}D^\h V^\h$, we always have a well-defined $\xi(Y)$. As in all earlier proofs, it remains to observe that the $\Sp(m)\times\Sp(n)$-invariance of $Y \sim\gis(m,n)$ passes on to $\xi(Y)$, forcing the distribution of $\xi(Y)$ to be the uniform distribution on $\V(n,\mathbb{H}^m)\times\Flag(\mathbb{H}^n)$, again a consequence of the uniqueness in Proposition~\ref{prop:unif}. The $\mathbb{R}$ and $\mathbb{C}$  cases are specializations of the result for $\mathbb{H}$.
\end{proof}
It follows immediately from Theorem~\ref{thm:gini} that we may obtain the uniform distribution on a Stiefel manifold by projecting the image of $\xi$.
\begin{corollary}[Uniform sampling on Stiefel manifold]\label{thm:rho}
Let $0 < k < n$ be integers and
\begin{alignat*}{2}
\rho :\mathbb{R}^{n\times k}_{\x} &\to\V(k,\mathbb{R}^n), \quad & Y &\mapsto UV^\tp, \\
\rho :\mathbb{C}^{n\times k}_{\x} &\to\V(k,\mathbb{C}^n), \quad & Y &\mapsto UV^\h, \\
\rho :\mathbb{H}^{n\times k}_{\x} &\to\V(k,\mathbb{H}^n), \quad & Y &\mapsto UV^\h.
\end{alignat*}
Then we have the following:
    \begin{enumerate}[{label=\upshape(\roman*)}]
        \item  If $Y \sim \gio(n,k)$, then $\rho(Y)\sim\unif\bigl(\V(k,\mathbb{R}^n)\bigr)$.
        \item  If $Y \sim \giu(n,k)$, then $\rho(Y)\sim\unif\bigl(\V(k,\mathbb{C}^n)\bigr)$.
        \item  If $Y \sim \gis(n,k)$, then $\rho(Y)\sim\unif\bigl(\V(k,\mathbb{H}^n)\bigr)$.
    \end{enumerate}    
\end{corollary}
    
\section{Cosine-sine vectors distributions of circular ensembles}\label{sec:circ}

Evidently, the uniform distributions in \eqref{eq:circ} fall into two types most clearly seen from the two hierarchies
\begin{equation}\label{eq:include}
\O(n) \subseteq \U(n) \subseteq \Sp(n) \qquad\text{and}\qquad \LGr(\mathbb{R}^{2n}) \subseteq \LGr(\mathbb{C}^{2n}) \subseteq \LGr(\mathbb{H}^{2n}).
\end{equation}
Here we will discuss the first trio of circular ensembles, with the second deferred to Section~\ref{sec:circ2}.
Our commentary in the beginning of Section~\ref{sec:gini} is particularly apt for these circular ensembles --- figuring out the ``correct'' matrix submanifold that parameterizes the left and right cosine-sine vectors pairs is the trickiest part of this section.

Circular ensembles are uniform distributions and thus their probability densities are constant functions entirely determined by their respective volume in Table~\ref{tab:vol}. Explicitly, the densities for $\cre(n)$, $\cue(n)$, $\cqe(n)$ are
\begin{equation}\label{eq:fcirc}
f_{\beta,n}^\C(Q) =\begin{cases}
  \frac{\gamma(1,n,1)}{\theta(2n,n(n+1))} &\text{if }\beta=1,\\[1ex]
  \frac{\gamma(1,n,2)}{\sqrt{n}\theta(n+1,2n(n+1))} &\text{if }\beta=2,\\[1ex]
  \frac{\gamma(1,n,4)}{\theta(2n,4n(n+1))} &\text{if }\beta=4,
\end{cases}
\end{equation}
with respect to the Riemannian volume form of $\O(n)$, $\U(n)$, and $\Sp(n)$ respectively.
 
\begin{proposition}[Cosine-sine vectors distributions of circular ensembles I]\label{prop:eta}
Let $0 < k < n$ be integers and
\begin{alignat*}{2}
\eta:\O_\x (n)&\to\bigl(\O(k)^2\times\O(n-k)^2\bigr)/\bigl(\O(1)^k\times\O(n-2k)\bigr), \quad &Q&\mapsto\lb(U_1,V_1,U_2,V_2)\rb_{\O}, \\
\eta:\U_\x (n)&\to\bigl(\U(k)^2\times\U(n-k)^2\bigr)/\bigl(\U(1)^k\times\U(n-2k)\bigr), \quad &Q&\mapsto\lb(U_1,V_1,U_2,V_2)\rb_{\U}, \\
\eta:\Sp_\x (n)&\to\bigl(\Sp(k)^2\times\Sp(n-k)^2\bigr)/\bigl(\Sp(1)^k\times\Sp(n-2k)\bigr), \quad &Q&\mapsto\lb(U_1,V_1,U_2,V_2)\rb_{\Sp},
\end{alignat*}
where the equivalence classes of left and right cosine-sine vectors are defined by
{\small
\begin{alignat*}{2}
\lb(U_1,V_1,U_2,V_2)\rb_{\O} & =\lb(U_1D,V_1D,U_2\diag(D,W),V_2\diag(D,W))\rb_{\O},\quad &  D &\in\O(1)^k, \; W\in\O(n-2k),\\
\lb(U_1,V_1,U_2,V_2)\rb_{\U} &=\lb(U_1D,V_1D,U_2\diag(D,W),V_2\diag(D,W))\rb_{\U},\quad &   D &\in\U(1)^k, \; W\in\U(n-2k),\\
\lb(U_1,V_1,U_2,V_2)\rb_{\Sp}&=\lb(U_1D,V_1D,U_2\diag(D,W),V_2\diag(D,W))\rb_{\Sp},\quad &   D &\in\Sp(1)^k, \; W\in\Sp(n-2k).
\end{alignat*}}%
Then we have the following:
\begin{enumerate}[{label=\upshape(\roman*)}]
    \item If $Q\sim\cre(n)$, then $\eta(Q)\sim\unif\bigl([\O(k)^2\times\O(n-k)^2]/[\O(1)^k\times\O(n-2k)]\bigr)$.
    \item If $Q\sim\cue(n)$, then $\eta(Q)\sim\unif\bigl([\U(k)^2\times\U(n-k)^2]/[\U(1)^k\times\U(n-2k)]\bigr)$.
    \item If $Q\sim\cqe(n)$, then $\eta(Q)\sim\unif\bigl([\Sp(k)^2\times\Sp(n-k)^2]/[\Sp(1)^k\times\Sp(n-2k)]\bigr)$.
    \end{enumerate}
\end{proposition}
\begin{proof}
We remind the reader of the cosine-sine decomposition in Table~\ref{tab:decomp}: For any $k \le n/2$, a matrix $Q \in \Sp(n)$ may be decomposed as
\[
Q=\begin{bmatrix}
    U_1&0\\
    0&U_2
\end{bmatrix}\begin{bmatrix}
    C&S&0\\
    -S&C&0\\
    0&0&I_{n-2k}
\end{bmatrix}\begin{bmatrix}
    V_1^\h&0\\
    0&V_2^\h
\end{bmatrix},
\]
where $U_1, V_1\in \Sp(k)$, $U_2, V_2 \in \Sp(n-k)$, and $C=\diag(c_1,\dots,c_k), S=\diag(s_1,\dots,s_k)$ with $c_i,s_i \in [0,1]$, $c_i^2+s_i^2=1$ for all $i=1,\dots,k$.

For a random matrix $Q\sim\cqe(n)$, $U_1,V_1,U_2,V_2$ are all random matrices and, with probability one, we may  assume that $c_1 < \dots < c_k$ and $s_1 > \dots > s_k$. For any $A_1,B_1\in\Sp(k)$, $A_2,B_2\in\Sp(n-k)$, the random matrix $\diag(A_1,A_2)Q\diag(B_1^\h,B_2^\h)$ is identically distributed as $Q$. It follows that cosine-sine vectors $\lb(A_1U_1,B_1V_1,A_2U_2,B_2V_2)\rb_{\Sp}$ are identically distributed as $\lb(U_1,V_1,U_2,V_2)\rb_{\Sp}$. Since $(A_1, A_2, B_1, B_2) \in \Sp(k)\times\Sp(k)\times\Sp(n-k)\times\Sp(n-k)$, this shows that the cosine-sine vectors distribution on the space
\[
[\Sp(k)\times\Sp(k)\times\Sp(n-k)\times\Sp(n-k)]/[\Sp(1)^k\times\Sp(n-2k)],
\]
is invariant under the action of $\Sp(k)\times\Sp(k)\times\Sp(n-k)\times\Sp(n-k)$. By Proposition~\ref{prop:unif}, there is only one such distribution, implying that
\[
\eta(Q)\sim\unif\bigl([\Sp(k)\times\Sp(k)\times\Sp(n-k)\times\Sp(n-k)]/[\Sp(1)^k\times\Sp(n-2k)]\bigr).
\]
The $\mathbb{R}$ and $\mathbb{C}$  cases are specializations of the result for $\mathbb{H}$.
\end{proof}
As we did with \eqref{eq:iso1} in the last section, here we would also like our parameter space for left and right cosine-sine vector pairs to be a matrix submanifold, with points that are actual matrices instead of equivalence classes. We accomplish this with the isomorphism
\[
[\Sp(k)^2\times\Sp(n-k)^2]/[\Sp(1)^k\times\Sp(n-2k)]\cong\Sp(k)^2\times\Sp(n-k)\times\Flag(1,\dots,k,\mathbb{H}^{n-k}),
\]
and with it we arrive at our main result of this section.

\begin{theorem}[Cosine-sine vectors distributions of circular ensembles II]\label{thm:zeta}
Let $0 < k < n$ be integers and
{\small
\begin{alignat*}{2}
\zeta:\O_\x (n)&\to\O(k)^2\times\O(n-k)\times\Flag(1,\dots,k,\mathbb{R}^{n-k}), \quad &Q&\mapsto(U_1(V_2)_{kk}^\tp,V_1(V_2)_{kk}^\tp,U_2V_2^\tp,V_2\Delta_aV_2^\tp), \\
\zeta:\U_\x (n)&\to\U(k)^2\times\U(n-k)\times\Flag(1,\dots,k,\mathbb{C}^{n-k}), \quad &Q&\mapsto(U_1(V_2)_{kk}^\h,V_1(V_2)_{kk}^\h,U_2V_2^\h,V_2\Delta_aV_2^\h), \\
\zeta:\Sp_\x (n)&\to\Sp(k)^2\times\Sp(n-k)\times\Flag(1,\dots,k,\mathbb{H}^{n-k}), \quad &Q&\mapsto(U_1(V_2)_{kk}^\h,V_1(V_2)_{kk}^\h,U_2V_2^\h,V_2\Delta_aV_2^\h),
\end{alignat*}}%
where $(V_2)_{kk}$ denotes the $k \times k$ leading principal submatrix of $V_2$. Then we have the following:
\begin{enumerate}[{label=\upshape(\roman*)}]
    \item If $Q\sim\cre(n)$, then $\zeta(Q)\sim\unif\bigl(\O(k)^2\times\O(n-k)\times\Flag(1,\dots,k,\mathbb{R}^{n-k})\bigr)$.
    \item If $Q\sim\cue(n)$, then $\zeta(Q)\sim\unif\bigl(\U(k)^2\times\U(n-k)\times\Flag(1,\dots,k,\mathbb{C}^{n-k})\bigr)$.
    \item If $Q\sim\cqe(n)$, then $\zeta(Q)\sim\unif\bigl(\Sp(k)^2\times\Sp(n-k)\times\Flag(1,\dots,k,\mathbb{H}^{n-k})\bigr)$.
    \end{enumerate}
\end{theorem}
\begin{proof}
Let $Q\sim\cqe(n)$. With probability one, $(U_1,V_1,U_2,V_2)$ is uniquely defined up to the equivalence relation
\[
\lb(U_1,V_1,U_2,V_2)\rb_{\Sp}=\lb(U_1D,V_1D,U_2\diag(D,A),V_2\diag(D,A))\rb_{\Sp}
\]
for any $D\in\Sp(1)^k$ and $A\in\Sp(n-2k)$. So $\zeta(Q)$ is well-defined as a random matrix. Similar to the proof of Proposition~\ref{prop:eta}, the $\Sp(n)\times\Sp(n)$-invariance of $Q$ leads us to the $\Sp(k)^2\times\Sp(n-k)^2$-invariance of $\zeta(Q)$, forcing the distribution of the latter to be the uniform distribution on $\Sp(k)^2\times\Sp(n-k)\times\Flag(1,\dots,k,\mathbb{H}^{n-k})$. The $\mathbb{R}$ and $\mathbb{C}$  cases are specializations of the result for $\mathbb{H}$.
\end{proof}
When $n$ is even and $k$ is chosen to be $n/2$, the results above take a particularly neat form:
\begin{corollary}[Cosine-sine vectors distributions of circular ensembles III]
Let $n=2k$ be a positive integer. Let $\eta$ and $\zeta$ be as in Proposition~\ref{prop:eta} and  Theorem~\ref{thm:zeta} respectively.
\begin{enumerate}[{label=\upshape(\roman*)}]
    \item If $Q\sim\cre(n)$, then $\eta(Q)\sim\unif\bigl(\O(k)^4/\O(1)^k\bigr)$ and $\zeta(Q)\sim\unif\bigl(\O(k)^3\times\Flag(\mathbb{R}^k)\bigr)$.
    \item If $Q\sim\cue(n)$, then $\eta(Q)\sim\unif\bigl(\U(k)^4/\U(1)^k\bigr)$ and $\zeta(Q)\sim\unif\bigl(\U(k)^3\times\Flag(\mathbb{C}^k)\bigr)$.
    \item If $Q\sim\cqe(n)$, then $\eta(Q)\sim\unif\bigl(\Sp(k)^4/\Sp(1)^k\bigr)$ and $\zeta(Q)\sim\unif\bigl(\Sp(k)^3\times\Flag(\mathbb{H}^k)\bigr)$.
    \end{enumerate}
\end{corollary}

\section{Autonne--Takagi vectors distributions of circular ensembles}\label{sec:circ2}

The Autonne--Takagi decompositions and Lagrangian Grassmannians will be important in this section. So before getting to the circular ensembles, we will need to fill in some gaps in the literature: (i) The standard Autonne--Takagi decomposition \cite[Corollaries~2.2.6(a) and 4.4.4(c)]{HJ} applies to complex symmetric matrix but we will need a version for Lagrangian and symplectically symmetric matrices. (ii) While the matrix submanifold representation for real Lagrangian Grassmannian  in the fourth column of Table~\ref{tab:rep} is known \cite[pp.~929--930]{Lag1}, i.e.,
\[
\LGr(\mathbb{R}^{2n}) \cong \U(n)/\O(n) \cong \{QQ^\tp : Q \in \U(n)\},
\]
we will need to establish the complex analogue and (part of) the quaternionic one.

In the case of the complex Lagrangian Grassmannian, there are definitions that deviate from the one we used \cite[pp.~79--80]{HL}. A notable example is Arnold's definition \cite{Arnold}, where $\LGr(\mathbb{C}^{2n})$ is taken to be the set of Lagrangian subspaces in $\mathbb{C}^{2n}$ with respect to the \emph{real} symplectic form in $\mathbb{R}^{4n}$, as opposed to the complex symplectic form in $\mathbb{C}^{2n}$. The result is that Arnold's complex Lagrangian Grassmannian is isomorphic to $\U(n)$ instead of $\Sp(n)/\U(n)$.

\subsection{Complex Lagrangian Grassmannian}\label{sec:CLG}

We begin with an Autonne--Takagi decomposition decomposition for $X \in \mathsf{V}^2(\mathbb{C}^{2n})$. Recall from \eqref{eq:J} that $I_{n,n}=\diag(I_n,-I_n)$.
\begin{proposition}[Lagrangian Autonne--Takagi decomposition]
Let $X\in\mathbb{C}^{2n \times 2n}$ with $X=X^\la$. Then there exist $Q\in\U(2n)$ and $\Sigma=\diag(\sigma_1,\dots,\sigma_{2n})\in\mathbb{R}^{2n\times 2n}$ such that \[
X=Q\Sigma Q^\la.
\]
\end{proposition}
\begin{proof}
    Let $Z=I_{n,n}X$. Then $Z^\h=X^\h I_{n,n}$. Since $X=X^\la$, we have $I_{n,n}X=I_{n,n}X^\la=X^\h I_{n,n}$, and thus $Z=Z^\h$. So $Z$ has an eigenvalue decomposition $Z=U\Lambda U^\h$ with $U\in\U(2n)$ and  diagonal $\Lambda \in \mathbb{R}^{2n \times 2n}$. Therefore,
\[
X=I_{n,n}Z=I_{n,n}U\Lambda U^\h=(I_{n,n}U)(\Lambda I_{n,n})(I_{n,n}U^\h).
\]
Now let $Q=I_{n,n}U$ and $\Sigma=\Lambda I_{n,n}$. Clearly, $Q\in\U(2n)$ and $\Sigma \in \mathbb{R}^{2n \times 2n}$ is diagonal. Furthermore, $Q^\la=I_{n,n}U^\h I_{n,n}I_{n,n}=I_{n,n}U^\h = Q$, as required.
\end{proof}

The next result will also be useful for Theorem~\ref{thm:omega}.
\begin{lemma}[$\U(n)$ as a subgroup of $\Sp(n)$]
Let $n$ be a positive integer. Then
\begin{align}
\U(n) &\cong\{Q\in\Sp(n):Q=I_{n,n}QI_{n,n}\}  \label{eq:fixpoint} \\
&= \{Q\in\Sp(n):Q^\la Q=I\}. \label{eq:U}
\end{align}
\end{lemma}
\begin{proof}
For convenience, we define the involution
\begin{equation}\label{eq:tau}
    \tau(Q)\coloneqq I_{n,n}QI_{n,n}
\end{equation}
for any $Q\in\mathbb{C}^{2n\times 2n}$. So $Q^\la = \tau(Q^\h)$ and the condition in \eqref{eq:fixpoint} is $\tau(Q) = Q$. Recall that $\Sp(n)=\U(2n)\cap\Sp(2n,\mathbb{C})$. Write $Q=\begin{bsmallmatrix}
      A&B\\
      C&D
  \end{bsmallmatrix} \in \Sp(n)$. Then $\tau(Q)=Q$ implies that $Q=\diag(A,D)$. As $Q\in \Sp(2n,\mathbb{C})$, we have $Q^\tp J_{2n} Q=J_{2n}$ and thus
\[
Q^\tp J_{2n}Q=\begin{bmatrix}
      0&A^\tp D\\
      -D^\tp A&0
  \end{bmatrix}.
\]
Hence $A^\tp D=I$. As $Q\in \U(2n)$, we have $A^\h A=I$ and thus $D=\Bar{A}$. Hence
\[
\{Q\in\Sp(n):Q=\tau(Q)\}\subseteq \{\diag(A,\Bar{A}): A\in\U(n)\} \cong\U(n).
\]
We verify that $\tau\diag\bigl((A,\Bar{A})\bigr) = \diag(A,\Bar{A})$ and so ``$\subseteq$'' may be replaced by ``$=$'' to give \eqref{eq:fixpoint}. It remains to show that for $Q\in\Sp(n)$, the conditions $Q=\tau(Q)$ and $Q^\la Q=I$ are equivalent. If $Q=\tau(Q)$, then
\[
Q^\la Q = I_{n,n}Q^\h I_{n,n} Q = I_{n,n} Q^\h I_{n,n} (I_{n,n} Q I_{n,n}) =I_{n,n} Q^\h Q I_{n,n}=I.
\]
If $Q^\la Q=I$, then
\[
Q^\h \tau(Q) =Q^\h I_{n,n}QI_{n,n} = I_{n,n} (I_{n,n} Q^\h I_{n,n}) QI_{n,n} = I_{n,n}  Q^\la Q I_{n,n} =I_{n,n}^2 = I.
\]
Hence we have \eqref{eq:U}.
\end{proof}

\begin{theorem}[Complex Lagrangian Grassmannian as matrix submanifold]\label{thm:CLGr}
For any positive integer $n$, we have
\begin{align}
\LGr(\mathbb{C}^{2n}) &\cong \Sp(n)/\U(n) \cong \{QQ^\la \in \mathbb{C}^{2n \times 2n} :Q\in\Sp(n)\} \label{eq:sub}\\
&= \Sp(n) \cap \mathsf{V}^2(\mathbb{C}^{2n}) = \{X\in\Sp(n) : X=X^\la\}. \notag
\end{align}
\end{theorem}
\begin{proof}
As an abstract manifold, $\LGr(\mathbb{C}^{2n})$ is the set of Lagrangian subspaces in $\mathbb{C}^{2n}$ with respect to its standard complex symplectic form. The first diffeomorphism may be found in \cite[Proposition~2.11]{Lag2}. We will establish the remaining two characterizations. Note that $\tau$ in \eqref{eq:tau} restricts to an involutive automorphism on $\Sp(n)$. By \eqref{eq:fixpoint}, $\{Q\in\Sp(n): Q=\tau(Q)\}\cong\U(n)$. Since $\Sp(n)/\U(n)$ and $\{Q\in\Sp(n):\tau(Q)=Q^{-1}\}$ are connected, by \cite[Proposition~12]{Cartan} we have
\[
\Sp(n)/\U(n)\cong \equalto{\{Q\in\Sp(n):\tau(Q)=Q^{-1}\}}{\{X\in\Sp(n) : X=X^\la\}} = \equalto{\{Q\tau(Q)^{-1} : Q\in\Sp(n)\}}{\{QQ^\la : Q\in\Sp(n)\}}
\]
as required.
\end{proof}

\subsection{Quaternionic Lagrangian Grassmannian}

To get an Autonne--Takagi decomposition for $X \in \mathsf{Y}^2(\mathbb{C}^{2n})$, we will need the following decomposition of complex skew-symmetric matrices, regarded variously as a skew-symmetric analog of Autonne--Takagi decomposition \cite[Corollary~2.2.6(b)]{HJ} or a special case of Youla decomposition \cite[p.~694]{Youla}. Recall from \eqref{eq:J} that $J_{2n} \coloneqq \begin{bsmallmatrix} 0 &  -I_n \\ I_n & 0 \end{bsmallmatrix}$.
\begin{lemma}\label{lem:you}
Let $X\in\mathbb{C}^{n\times n}$ with $X^\tp = - X$. Then there exists $U\in\U(n)$ such that
\[
X=U\diag(\sigma_1J_2,\dots,\sigma_kJ_2,0,\dots,0) U^\tp,
\]
where $\sigma_1,\dots,\sigma_k \in \mathbb{R}_\p$ are the nonzero singular values of $X$ and $J_2 = \begin{bsmallmatrix}
    0&-1\\
    1&0
\end{bsmallmatrix}$.
\end{lemma}

We emphasize that the following symplectic Autonne--Takagi decomposition differs from both the quaternionic Autonne--Takagi decomposition in \cite{QAT} and the QDQ decomposition in \cite[p.~16]{EJ}.
\begin{proposition}[Symplectic Autonne--Takagi decomposition]\label{ATS}
Let $X\in\mathbb{C}^{2n\times 2n}$ with $X=X^\s$. Then there exist $Q\in\U(2n)$ and $\Sigma=\diag(\sigma_1,\dots,\sigma_{2n})\in\mathbb{R}^{2n\times 2n}_\p$ such that
\[
X=Q\Sigma Q^\s.
\]
\end{proposition}
\begin{proof}
Let $Z=J_{2n}X$. Then
\[
Z^\tp =X^\tp J_{2n}^\tp=(-J_{2n}J_{2n})X^\tp (-J_{2n})=-J_{2n}(-J_{2n}X^\tp J_{2n})=-J_{2n}X^\s=-J_{2n}X=-Z.
\]
By Lemma~\ref{lem:you}, there exist $U\in\U(2n)$ and $\sigma_1,\dots,\sigma_k>0$ so that
\[
Z=U\diag(\sigma_1J_2,\dots,\sigma_kJ_2,0\dots,0)U^\tp.
\]
Observe that
\[
\diag(\sigma_1J_2,\dots,\sigma_kJ_2,0\dots,0)=\diag(\sigma_1,\sigma_1,\dots,\sigma_k,\sigma_k,0,\dots,0)\diag(J_2,\dots,J_2),
\]
and that there exists a permutation matrix $P$ with $\diag(J_2,\dots,J_2)=PJ_{2n}P^\tp$. Hence,
\begin{align}
X &=J_{2n}^\tp Z=J_{2n}^\tp U\diag(\sigma_1,\sigma_1,\dots,\sigma_k,\sigma_k,0,\dots,0)PJ_{2n}P^\tp U^\tp \nonumber \\
&=(J_{2n}^\tp UP)\bigl(P^\tp\diag(\sigma_1,\sigma_1,\dots,\sigma_k,\sigma_k,0,\dots,0)P\bigr)(J_{2n}P^\tp U^\tp). \label{eq:last}
\end{align}
Let $Q \coloneqq J_{2n}^\tp UP$. Then
\[
Q^\s=-J_{2n}P^\tp U^\tp J_{2n}J_{2n}=J_{2n}P^\tp U^\tp,
\]
and clearly $Q\in\U(2n)$. Let $\Sigma \coloneqq P^\tp\diag(\sigma_1,\sigma_1,\dots,\sigma_k,\sigma_k,0,\dots,0)P$. Noting that conjugation by a permutation matrix simply permutes the diagonal entries, we see that $\Sigma=\diag(\sigma_1,\dots,\sigma_{2n})\in\mathbb{R}^{2n\times 2n}_\p$ after relabeling of the diagonal entries. Hence the last term in \eqref{eq:last} takes the form $Q\Sigma Q^\s$.
\end{proof}

Using the symplectic Autonne--Takagi decomposition, we deduce the description of a quaternionic Lagrangian Grassmannian in the last column of Table~\ref{tab:rep}.
\begin{corollary}[Quaternionic Lagrangian Grassmannian as matrix submanifold]\label{cor:QLGr}
For any positive integer $n$, we have
\begin{align}
\LGr(\mathbb{H}^{2n}) &\cong \U(2n)/\Sp(n) \cong \{QQ^\s \in \mathbb{C}^{2n \times 2n} : Q\in\U(2n)\} \notag \\
&= \U(2n) \cap \mathsf{Y}^2(\mathbb{C}^{2n}) =  \{X\in\U(2n) :X=X^\s\}. \label{eq:2}
\end{align}
\end{corollary}
\begin{proof}
As an abstract manifold, $\LGr(\mathbb{H}^{2n})$ is the set of Lagrangian subspaces in $\mathbb{H}^{2n}$ with respect to its standard quaternionic skew-Hermitian form. The first and second isomorphisms may be found in \cite[Proposition~2.32]{Lag2} and \cite[Section~1]{Lag3}. It remains to show the alternate characterization in \eqref{eq:2}.  Let $X\in\U(2n)$ with $X=X^\s$. By Proposition~\ref{ATS}, we must have $X=Q\Sigma Q^\s$ for some $Q\in\U(2n)$ and $\Sigma=\diag(\sigma_1,\dots,\sigma_{2n})\in\mathbb{R}_\p^{2n\times 2n}$. But since $X\in\U(2n)$, we must also have that $\Sigma \in\U(2n)$, which is only possible if $\Sigma=I$. Hence $X=QQ^\s$. Conversely, for any $Q\in\U(2n)$, we clearly have that $QQ^\s=(QQ^\s)^\s$, and thus $QQ^\s\in  \U(2n) \cap \mathsf{Y}^2(\mathbb{C}^{2n})$.
\end{proof}

The next result will be useful for Theorem~\ref{thm:omega}.
\begin{lemma}[$\Sp(n)$ as a subgroup of $\U(2n)$]\label{lem:sp}
For any positive integer $n$, we have
\[
\Sp(n) = \{Q\in\U(2n):Q^\s Q=I\}.
\]
\end{lemma}
\begin{proof}
As $\Sp(n)=\U(2n)\cap\Sp(2n,\mathbb{C})$, we need to show that the conditions $Q^\tp J_{2n}Q=J_{2n}$ and $Q^\s Q=I$ are equivalent. Suppose $Q^\tp J_{2n}Q=J_{2n}$, then
\[
Q^\s Q=-J_{2n}Q^\tp J_{2n}Q=-J_{2n}^2 =I.
\]
Conversely, suppose $Q^\s Q=I$, then $-J_{2n}Q^\tp J_{2n}Q = I$. So  $Q^\tp J_{2n}Q=-J_{2n}^{-1}=J_{2n}$.
\end{proof}

\subsection{Distributions on Lagrangian Grassmannians}\label{sec:circ3}

As in Section~\ref{sec:circ}, the densities of $\coe(n)$, $\cle(n)$, $\cse(n)$ are constant functions determined by their respective volume in Table~\ref{tab:vol}. Explicitly,
\[
g_{\beta,n}^\C(Q) =\begin{cases}
\frac{\gamma(1,n,2)}{\sqrt{n}\theta(1-n,n(n+1))\gamma(1,n,1)}  &\text{if }\beta=1,\\[1ex]
 \frac{\sqrt{n}\gamma(1,n,4)}{\theta(n-1,2n(n+1))\gamma(1,n,2)} &\text{if }\beta=2,\\[1ex]
\frac{\gamma(1,2n,2)}{\sqrt{2n}\theta(1,4n^2)\gamma(1,n,4)}  &\text{if }\beta=4,
\end{cases}
\]
with respect to the Riemannian volume form of $\U(n)/\O(n)$, $\Sp(n)/\U(n)$, and $\U(2n)/\Sp(n)$ respectively. To the best of our knowledge, the space $\Sp(n)/\U(n)$ made a passing appearance in \cite[p.~19]{EJ} but is otherwise not studied in the random matrix literature. For easy reference, we call it the circular Lagrangian ensemble in this article.

\begin{theorem}[Autonne--Takagi vectors distributions of circular ensembles]\label{thm:omega}
Let $n$ be a positive integer and
\begin{alignat*}{2}
\omega:\LGr(\mathbb{R}^{2n})&\to\U(n)/\O(n), \quad &X&\mapsto\lb Q\rb_{\O}, \\
\omega:\LGr(\mathbb{C}^{2n})&\to\Sp(n)/\U(n), \quad &X&\mapsto\lb Q\rb_{\U}, \\
\omega:\LGr(\mathbb{H}^{2n})&\to\U(2n)/\Sp(n), \quad &X&\mapsto\lb Q\rb_{\Sp},
\end{alignat*}
where the equivalence classes of Autonne--Takagi vectors are defined by
\begin{alignat*}{2}
\lb Q\rb_{\O} & =\lb QW\rb_{\O},\quad &  W &\in\O(n),\\
\lb Q\rb_{\U} & =\lb QW\rb_{\U},\quad &  W &\in\U(n),\\
\lb Q\rb_{\Sp} & =\lb QW\rb_{\Sp},\quad &  W &\in\Sp(n).
\end{alignat*}
Then we have the following:
    \begin{enumerate}[{label=\upshape(\roman*)}]
    \item If $X\sim\coe(n)$, then $\omega(X)\sim\unif(\U(n)/\O(n))$.
    \item If $X\sim\cle(n)$, then $\omega(X)\sim\unif(\Sp(n)/\U(n))$.
    \item If $X\sim\cse(n)$, then $\omega(X)\sim\unif(\U(2n)/\Sp(n))$.  
\end{enumerate}
\end{theorem}
\begin{proof}
We begin with the $\cse(n)$ case. Using the characterization in Corollary~\ref{cor:QLGr}, we let $QQ^\s=ZZ^\s$ be two Autonne--Takagi decompositions of $X \in \LGr(\mathbb{H}^{2n})$. Then as $(Z^\h Q)(Z^\h Q)^\s=I$, we have that $Z^\h Q\in\Sp(n)$ by Lemma~\ref{lem:sp}. Hence $\omega(X)$ is well-defined. Now let $X\sim\cse(n)$. The density of $\cse(n)$, being constant, is trivially invariant under any group action. So $X$ and $AXA^\s$ are identically distributed for any $A\in\U(2n)$. Hence so are $\omega(X)$ and $\omega(AXA^\s)$. Let $X = QQ^\s$ be an Autonne--Takagi decomposition, noting that $Q$ is now a random matrix. Then
\[
\omega(AXA^\s) =\omega\bigl(AQ(AQ)^\s\bigr) = \lb AQ\rb_{\Sp} =  A \lb Q\rb_{\Sp} = A\omega(X).
\]
So $\omega(X)$ and $A\omega(X)$ are identically distributed. Since this holds for all $A \in \U(2n)$, it follows from Proposition~\ref{prop:unif} that $\omega(X)\sim\unif(\U(2n)/\Sp(n))$.

The $\cle(n)$ case is similar, with  Corollary~\ref{thm:CLGr} taking the place of Corollary~\ref{cor:QLGr} and  \eqref{eq:U} taking the place of Lemma~\ref{lem:sp}, together with corresponding replacement of $X^\s$ by $X^\la$ and $\U(2n)/\Sp(n)$ by $\Sp(n)/\U(n)$. The $\coe(n)$ case is also similar.
\end{proof}

We conclude with a brief mention of the relations between the circular ensembles in this section with those in Section~\ref{sec:circ}. 
\begin{proposition}\label{prop:rel}
Let $n$ be a positive integer. Then we have the following:
\begin{enumerate}[{label=\upshape(\roman*)}]
\item If $Q\sim\cue(n)$, then $Q Q^\tp \sim\coe(n)$.
\item\label{it:cle} If $Q\sim\cqe(n)$, then $Q Q^\la \sim\cle(n)$.
\item If $Q\sim\cue(2n)$, then $Q Q^\s \sim\cse(n)$.
\end{enumerate}
\end{proposition}
\begin{proof}
We will just prove \ref{it:cle} as the other two cases are well-known \cite[Section~1]{Lag3}. The group $\Sp(n)$ acts on the submanifold characterization of $\LGr(\mathbb{C}^{2n})$ in \eqref{eq:sub} via $(A,QQ^\la)\mapsto AQQ^\la A^\la $ for any $A\in\Sp(n)$ and $QQ^\la \in \LGr(\mathbb{C}^{2n})$. Let $Q\sim\cqe(n)$. Then $AQ$ and $Q$ are identically distributed for any $A\in\Sp(n)$. Thus $AQQ^\la A^\la$ and $QQ^\la$ are also identically distributed on $\LGr(\mathbb{C}^{2n})$. Since this holds for all $A\in\Sp(n)$, we have $QQ^\la\sim\cle(n)$ by Proposition~\ref{prop:unif}.
\end{proof}

\section{Conclusion}

The map that takes a matrix to its matrix of eigen, singular, cosine-sine, or Autonne--Takagi vectors is often a $G$-equivariant map between well-known spaces.  By representing all spaces as matrix submanifolds, we may view such a map as taking random matrices to random matrices, mapping well-known random matrix ensembles to the $G$-invariant distributions on manifolds of geometric importance.

Another observation is that the distributions in the last column of Table~\ref{tab:sum} are all uniform distributions. This leads us to the interpretation that the matrix decompositions in Table~\ref{tab:decomp} resolve the randomness in these random matrix ensembles into two parts: a uniform part captured by the eigen, singular, cosine-sine, or Autonne--Takagi \emph{vectors}; and a non-uniform part captured by the eigen, singular, cosine-sine, or Autonne--Takagi \emph{values}.

The maps in this article are all univariate maps. For example,  $\varphi_{1,\dots,n-1}$ takes one random matrix $X \sim \goe(n)$ to another random matrix $\varphi_{1,2,\dots,n-1}(X)\sim\unif\bigl(\Flag(\mathbb{R}^n)\bigr)$. We might ask about the multivariate analogue, in the spirit of free probability. We leave this to future work for interested readers.

\subsection*{Acknowledgment} We would like to thank Thomas Rongbiao Wang for very helpful discussions regarding the Lagrangian Grassmannians.

\bibliographystyle{abbrv}

\appendix

\section{Alternative definitions of the ensembles}\label{app:alt}

Let $\nor_{\mathbb{R}}(0,\sigma^2)$, $\nor_{\mathbb{C}}(0,\sigma^2)$, and $\nor_{\mathbb{H}}(0,\sigma^2)$ denote the real, complex, and quaternionic normal distributions respectively, i.e., with densities
\[
f_{\beta,\sigma^2}(y)
=
\begin{cases}
\frac{1}{\sqrt{2\pi\sigma^2}}
  e^{-y^2/2\sigma^2}
  & \text{if }\beta=1,\; y\in\mathbb{R},\\[1.5ex]
\frac{1}{\pi\sigma^2}
  e^{-|y|^2/\sigma^2}
  & \text{if }\beta=2,\; y\in\mathbb{C},\\[1.5ex]
\frac{4}{\pi^2\sigma^4}
 e^{-2|y|^2/\sigma^2}
  & \text{if }\beta=4,\; y\in\mathbb{H},
\end{cases}
\]
respectively. One way to view Gaussian, Jacobi, Laguerre, and Gini ensembles is that they are random matrices that can be generated from the univariate normal distribution:
\begin{enumerate}[{label=\upshape(\roman*)}]
\item $X\sim \goe(n)$ iff $X$ is a real symmetric $n \times n$ random matrix with independent entries
\[
x_{ii}\sim \nor_{\mathbb{R}}(0,1),\quad1\le i\le n;\qquad
x_{ij}\sim \nor_{\mathbb{R}}\Bigl(0,\frac{1}{2}\Bigr), \quad 1\le i<j\le n.
\]

\item $X\sim \gue(n)$ iff $X$ is a complex Hermitian $n \times n$  random matrix with independent entries
\[
x_{ii}\sim \nor_{\mathbb{R}}(0,1),\quad 1\le i\le n;\qquad
x_{ij}\sim \nor_{\mathbb{C}}(0,1), \quad1\le i<j\le n.
\]

\item $X\sim \gse(n)$ iff $X$ is a quaternionic self-dual $n \times n$  random matrix with independent entries
\[
x_{ii}\sim \nor_{\mathbb{R}}(0,1), \quad1\le i\le n;\qquad
x_{ij}\sim \nor_{\mathbb{H}}(0,2),\quad 1\le i<j\le n.
\]

\item $Y\sim\gio(m,n)$ iff $Y$ is a real $m \times n$ random matrix with independent entries
\[
y_{ij}\sim \nor_{\mathbb{R}}(0,1), \quad 1 \le i \le m,\; 1 \le j \le n.
\]

\item $Y\sim\giu(m,n)$ iff $Y$ is a complex $m \times n$ random matrix with independent entries
\[
y_{ij}\sim \nor_{\mathbb{C}}(0,1), \quad 1 \le i \le m,\; 1 \le j \le n.
\]

\item $Y\sim\gis(m,n)$ iff $Y$ is a quaternionic $m \times n$ random matrix with independent entries
\[
y_{ij}\sim \nor_{\mathbb{H}}(0,1), \quad 1 \le i \le m,\; 1 \le j \le n.
\]
\end{enumerate}
The Laguerre ensembles may be defined in terms of Ginibre ensembles. Let $m \ge n$. Then
\begin{enumerate}[resume*]
\item $H\sim\loe(m,n)$ iff $H$ is a real symmetric $n \times n$ random matrix given by
\[
H=Y^\tp Y
\]
with  $Y \sim\gio(m,n)$.

\item $H\sim\lue(m,n)$ iff $H$ is a complex Hermitian $n \times n$ random matrix given by
\[
H=Y^\h Y
\]
with $Y\sim\giu(m,n)$.

\item $H\sim\lse(m,n)$ iff $H$ is a quaternionic self-dual $n \times n$ random matrix given by
\[
H=Y^\h Y
\]
with $Y\sim\gis(m,n)$.
\end{enumerate}
The Jacobi ensembles may be defined in terms of Ginibre ensembles. Let $l \ge n$ and $m \ge n$. Then
\begin{enumerate}[resume*]
\item $H\sim\joe(l,m,n)$ iff $H$ is a real symmetric $n \times n$ random matrix given by
\[
H=(Y^\tp Y+Z^\tp Z)^{-1/2}Z^\tp Z(Y^\tp Y+Z^\tp Z)^{-1/2}
\]
with independent $Y\sim\gio(m,n)$ and $Z\sim\gio(l,n)$.

\item $H\sim\jue(l,m,n)$ iff $H$ is a complex Hermitian $n \times n$ random matrix given by
\[
H=(Y^\h Y+Z^\h Z)^{-1/2}Z^\h Z(Y^\h Y+Z^\h Z)^{-1/2}
\]
with independent $Y\sim\giu(m,n)$ and $Z\sim\giu(l,n)$.

\item $H\sim\jse(l,m,n)$ iff $H$ is a quaternionic self-dual $n \times n$ random matrix given by
\[
H=(Y^\h Y+Z^\h Z)^{-1/2}Z^\h Z(Y^\h Y+Z^\h Z)^{-1/2}
\]
with independent $Y\sim\gis(m,n)$ and $Z\sim\gis(l,n)$.
\end{enumerate}
While the circular real, complex, and quaternionic ensembles do not have entrywise descriptions in terms of the univariate normal distribution, it follows from Proposition~\ref{prop:rel} that the circular orthogonal, Lagrangian, and symplectic ensembles may be defined in terms of the circular unitary and quaternionic ensembles.
\begin{enumerate}[resume*]
\item $Z\sim\coe(n)$ iff $Z$ is a complex symmetric unitary $n\times n$ random matrix given by
\[
Z=Q Q^\tp
\]
with $Q\sim\cue(n)$.
\item $Z\sim\cle(n)$ iff $Z$ is a complex Lagrangian symmetric compact symplectic $2n\times 2n$ random matrix given by
\[
Z=Q Q^\la
\]
with $Q\sim\cqe(n)$.
\item $Z\sim\cse(n)$ iff $Z$ is a complex symplectically symmetric unitary $2n\times 2n$ random matrix given by
\[
Z=Q Q^\s
\]
with $Q\sim\cue(2n)$.
\end{enumerate}

\end{document}